\documentclass{amsart}

\usepackage{amsmath,amssymb,amsfonts,amsthm,mathtools,mathrsfs}
\usepackage[mathscr]{eucal}
\usepackage{cases,empheq,relsize}
\usepackage{color,graphicx,epstopdf,subfigure,tikz}
\usetikzlibrary{arrows,automata,backgrounds,patterns,calc}
\usepackage{pgfplots}
\usepackage[draft]{todonotes}
\usepackage{array,tabularx,varwidth,enumerate,verbatim,xparse}
\usepackage[normalem]{ulem}
\usepackage[linkbordercolor=blue]{hyperref}
\hypersetup{
  colorlinks=false,
  linkcolor=blue,
  citecolor=red,
  filecolor=magenta,
  urlcolor=cyan
}

\numberwithin{equation}{section}
\theoremstyle{plain}
\newtheorem{thm}{Theorem}[section]
\newtheorem{cor}[thm]{Corollary}
\newtheorem{lem}[thm]{Lemma}
\newtheorem{prop}[thm]{Proposition}
\theoremstyle{remark}
\newtheorem{rmk}{Remark}

\newcommand{\supp}{\text{supp\! }}
\newcommand{\dist}{\text{dist\! }}
\newcommand{\R}{\mathbb R}
\newcommand{\Z}{\mathbb Z}
\newcommand{\C}{\mathbb C}
\newcommand{\N}{\mathbb N}
\newcommand{\CK}{\mathcal K}
\newcommand{\fa}{\mathfrak a}

\author[Jeong]{Eunhee Jeong}
\address[Eunhee Jeong]{
  Department of Mathematics Education, and Institute of Pure and Applied Mathematics, Jeonbuk National University, Jeonju 54896, Republic of Korea}
\email{eunhee@jbnu.ac.kr}

\author[Ko]{Hyerim Ko}
\address[Hyerim Ko]{
  Department of Mathematics, and Institute of Pure and Applied Mathematics, Jeonbuk National University, Jeonju, 54896, Republic of Korea}
\email{kohr@jbnu.ac.kr}

\author[Shiraki]{Shobu Shiraki}
\address[Shobu Shiraki]{
  Department of Mathematics, Faculty of Science, University of Zagreb, Bijeni\v{c}ka cesta 30, 100000 Zagreb, Croatia}
\email{shobu.shiraki@math.hr}

\begin{document}
\title[Orthonormal maximal estimates with quadratic potentials]
{
Maximal estimates of the Schr\"odinger operators with quadratic potential for orthonormal systems
}

\keywords{Schr\"odinger equation, Hermite operator, special Hermite operator, orthonormal system, pointwise convergence}
\subjclass[2010]{35B45 (primary); 35P10, 35Q41(secondary).}
\begin{abstract}
We establish an equivalence principle for local mixed-norm estimates for generalized Schr\"odinger propagators acting on orthonormal systems. Under suitable mixed-norm conditions, estimates for generalized propagators follow from the corresponding free Schr\"odinger estimates for every summability exponent $1\le\beta\le2$.
Applications include maximal estimates for Schr\"odinger equations with time-dependent quadratic potentials and for the special Hermite operator, as well as almost everywhere convergence of the associated densities. For the special Hermite operator, we also obtain a single-function maximal estimate that is sharp up to the endpoint.
We construct counterexamples for the Schr\"odinger operator, by combining single particle and many particle constructions in complementary dimensions, yielding new necessary conditions in an intermediate range of  summability exponents.
\end{abstract}

\maketitle

\section{Introduction}

\subsection{Schr\"odinger operators with quadratic potentials}
Let $d\ge1$ and $s\in \R$. For a self-adjoint operator $\mathcal A$ on $L^2(\R^d)$,
we write $u(x,t)=e^{-it\mathcal A}f(x)$ for the solution to the evolution equation
\begin{align}\label{dispersive eq}
i\partial_tu-\mathcal A u=0, \quad u(x,0)=f(x).
\end{align}
Throughout, $H^s(\R^d)$ and $\dot{H}^s(\R^d)$ denote the standard $L^2$-based inhomogeneous and homogeneous Sobolev spaces, respectively.

Motivated by Carleson's pointwise convergence problem for the free Schr\"odinger equation, we are interested in determining the smallest regularity $s\in \R$ for which
\begin{align}\label{pt-single}
\lim_{t\rightarrow0+} e^{-it\mathcal A}f(x) = f(x), \quad \text{a.e.}\,\,x\in\R^d
\end{align}
holds for all $f\in H^s(\R^d)$.
A standard argument shows that such pointwise convergence follows from the maximal estimate
\begin{align}\label{max1}
\big\| \sup_{t\in I_{\mathcal A}} |e^{-it\mathcal A}f| \big\|_{L^2(B_1^d)} \lesssim \|f\|_{H^s(\R^d)}
\end{align}
for the corresponding $s$. Here, $B^d_\rho(x)$ denotes the ball in $\R^d$ of radius $\rho>0$ centered at $x\in \R^d$. When the ball is centered at the origin, we simply write $B_\rho^d$. 

For the Schr\"odinger operator $\mathcal A=-\Delta$, we set $I=I_{-\Delta}=[0,1]$. 
For the other operators considered below, we take $I_{\mathcal A}=[0,T_{\mathcal A}]$, where $T_{\mathcal A}>0$ is chosen according to the operators.

When $\mathcal A=-\Delta$, we recover the free Schr\"odinger operator. In $d=1$,
Carleson \cite{Carleson} proved convergence for $s\geq\frac14$, and Dahlberg--Kenig
\cite{DK82} showed that this is sharp.
In higher dimensions, the problem is considerably more difficult (we refer the reader to \cite{Lee06, Bourgain12, LR17, DGLZ} for the history of the problem and related results).
Bourgain's breakthrough construction \cite{Bourgain16} established the necessary condition $s\ge \frac{d}{2(d+1)}$. The corresponding estimate for every $s>\frac{d}{2(d+1)}$ was subsequently proved by Du--Guth--Li \cite{DGL17} in dimension $d=2$ and by Du--Zhang \cite{DZ19} in higher dimensions.
\begin{thm}[{\cite{Bourgain16, DGL17, DZ19}}]\label{t:max schrodinger single}
Let $d\ge2$. Then
\begin{equation*}
\Big\| \sup_{t \in I} \big| e^{it\Delta} f \big|\Big\|_{L^2(B_1^d)}
\lesssim
\|f\|_{H^s(\R^d)}
\end{equation*}
holds for all $f\in H^s(\R^d)$ whenever $s>\frac{d}{2(d+1)}$. Conversely, if this estimate holds, then $s\geq\frac{d}{2(d+1)}$.
\end{thm}

The free Schr\"odinger operator serves as the model for the more general operators considered in this paper.
Following \cite{CLV}, we also consider Schr\"odinger equations with time dependent quadratic potentials.
Let $\omega = (\omega_1,\dots, \omega_d) \in C^1(\R)$ be real-valued function. Define the time-dependent Hamiltonian
\begin{equation*}
\mathcal H_{\omega} = -\frac12\Big(\Delta -\sum_{j=1}^d \omega_j(t)x_j^2\Big).
\end{equation*}
When $\omega_j\equiv1$ for all $j=1,\dots,d$, the potential is time independent and $2\mathcal H_{\omega}$ becomes the Hermite operator
\[
\mathcal{H}
=-\Delta+|x|^2
\]
where $x=(x_1,\dots, x_d)$. 
The Hermite operator $\mathcal H$ has discrete spectrum $2\mathbb N_0+d$, where $\mathbb N_0=\mathbb N\cup\{0\}$.
The pointwise convergence problem for $e^{-it\mathcal H}f$ was studied by Cowling \cite{Cowling}, Yajima \cite{Yajima}, and Sj\"ogren--Torrea \cite{ST}. In particular, Sj\"ogren--Torrea established a pointwise identity\footnotemark relating $e^{-it \mathcal H}f$ to the Schr\"odinger propagator $e^{it\Delta}f$. Using this identity, Lee--Rogers \cite{LR} established an equivalence between the corresponding local maximal estimates \eqref{max1} for the Hermite operator $\mathcal H$ and Schr\"odinger operator.  

\footnotetext{More precisely, Sj\"ogren--Torrea \cite{ST} proved that $e^{-i\frac{\arctan t}{2}\mathcal H}f(x)=e^{-\frac{it}2|x|^2}(1+t^2)^{d/4}e^{i\frac t2\Delta}f(\sqrt{1+t^2}x)$ for $t\in(0,\infty)$.}

Cho--Lee--Vargas \cite{CLV} extended this equivalence to time-dependent quadratic potentials. On a sufficiently small time interval $[0,T_{\mathcal H_\omega}]$, the local maximal estimates for $e^{-it\mathcal H_\omega}$ hold for $s>s_0$ if and only if the corresponding Schr\"odinger estimates hold for the range $s>s_0$.

\subsection{Maximal estimates for the special Hermite operator}
We also consider the special Hermite operator, or twisted Laplacian, on $\C^d\cong\R^{2d}$, defined by
\[
\mathcal L = -\sum_{j=1}^d\Big[ \Big(\frac{\partial}{\partial x_j} -\frac12 iy_j \Big)^2 +\Big(\frac{\partial}{\partial y_j}+\frac12 ix_j\Big)^2\Big]
\]
where $x,y \in \R^d$, and we write $z=(x,y)\in\mathbb R^{2d}$.
The operator $\mathcal L$ is self-adjoint and has the same discrete spectrum $2\mathbb N_0+d$ as the Hermite operator $\mathcal H$.
In what follows, we consider the operator $e^{-it\mathcal L}f$, which is the solution to \eqref{dispersive eq} associated with $\mathcal A=\mathcal L$.

We first present a single-function maximal estimate for the special Hermite operator, which appears to be new.
\begin{thm}\label{thm:twisted-single}
Let $d\ge1$ and $n=2d$. Fix $0<T_{\mathcal L}<\pi/2$ and set $I_{\mathcal L}=[0,T_{\mathcal L}]$.
Then
\begin{equation*}
    \big\| \sup_{t\in I_{\mathcal L}} \big| e^{-it\mathcal L} f \big| \big\|_{L_z^2(B_1^n)}
    \lesssim
    \|f\|_{H^s(\R^n)}
\end{equation*}
holds for all $f\in H^s(\R^n)$ with $s>\frac{n}{2(n+1)}$.
\end{thm}

By the standard density argument described above, Theorem \ref{thm:twisted-single} implies \eqref{pt-single} with $\mathcal A=\mathcal L$ and $d$ replaced by $n$, for every $f\in H^s(\R^n)$ with $s>\frac{n}{2(n+1)}$.

To prove Theorem \ref{thm:twisted-single}, we derive the pointwise identity \eqref{id-special} between the special Hermite and free Schr\"odinger propagators, analogous to the relation of
Sj\"ogren--Torrea \cite{ST} for the Hermite operator.
By the equivalence theorem of \cite{CLV}, this identity gives an equivalence between the local mixed-norm estimates for the two propagators. The sufficient bound in Theorem \ref{thm:twisted-single} then follows from Theorem \ref{t:max schrodinger single}. Conversely, this equivalence and Bourgain's counterexample \cite{Bourgain16} give the necessary condition $s\ge\frac{n}{2(n+1)}$. Thus the regularity is sharp up to the endpoint.

The same pointwise identity also yields an exact equality between the corresponding Strichartz norms on the Schr\"odinger-admissible line. See Proposition \ref{prop:sp-stri}.

\subsection{Orthonormal generalization}
The maximal estimates discussed so far concern a single initial datum $f$ and may be viewed as single-particle estimates. The dynamics of infinitely many fermions naturally lead to estimates for orthonormal systems of initial data \cite{LS14, LS15, CHP17, CHP18}. We therefore consider orthonormal analogues of the preceding maximal estimates. 
For the free Schr\"odinger operator, we are interested in estimates of the form 
\begin{align}\label{sup-schr}
\Big\| \sup_{t\in I} \Big| \sum_j \lambda_j |e^{it\Delta}f_j|^2 \Big| \Big\|_{L_x^{\frac q2}(B_1^d)}
\lesssim \|\lambda\|_{\ell^\beta}
\end{align}
for $q\in[2,\infty]$ and $\beta\in[1,\infty]$. Here, $(f_j)$ forms an orthonormal system in $H^s(\R^d)$ and $\lambda=(\lambda_j)$ is a sequence of complex numbers.
Thus $(\langle D\rangle^s f_j)$ is orthonormal in $L^2(\R^d)$. For $\beta=1$, the triangle inequality reduces \eqref{sup-schr} to the corresponding single function maximal estimate in $L_x^q(B_1^d)$, which is \eqref{max1} with $\mathcal A=-\Delta$ when $q=2$.
Thus the estimate \eqref{sup-schr} is of particular interest when $\beta>1$,
and one aims to determine the largest value of $\beta$ for each $s$.

Motivated by the pointwise convergence problem for density functions associated
with infinitely many fermions, which will be discussed in the next subsection,
Bez--Lee--Nakamura \cite{BLN_Selecta} established the following sharp estimate at the critical regularity
$s=\frac14$ in one dimension: 
for every orthonormal system $(f_j)\subset \dot{H}^{\frac14}(\R)$ and every sequence $\lambda=(\lambda_j)\in \ell^\beta$, the estimate
\begin{align*}
\Big\| \sup_{t\in\R} \Big| \sum_j \lambda_j |e^{it\partial_x^2}f_j|^2 \Big| \Big\|_{L_x^{2,\infty}(\R)}
\lesssim \|\lambda\|_{\ell^\beta}
\end{align*}
holds if and only if $\beta<2$.
Bez, Kinoshita and the third author \cite{BKS_TLMS} subsequently established corresponding
strong-type local maximal estimates for $s\in[\frac d4,\frac d2)$ in higher dimensions.

For the applications below, we restrict attention to the diagonal case $q=2\beta$.
For $1\le \beta \le \infty$, set
\begin{align}\label{s0nbeta}
s(d,\beta)=\max\Big\{ \frac d2\Big(1-\frac1\beta\Big), \,\frac{d}{2(d+1)}\Big(d-\frac{d-1}\beta\Big)\Big\}.
\end{align}

\begin{thm}[{\cite{BKS_TLMS, Carleson, DGL17, DZ19}}]\label{thm:max}
Let $d\geq1$, $\beta\in [1,\infty]$, and $s>s(d,\beta)$, and set $q=2\beta$. Then \eqref{sup-schr}
holds for all orthonormal systems $(f_j)\subset H^s(\R^d)$ and all sequences $\lambda =(\lambda_j)\in \ell^\beta$.
\end{thm}

For $d\ge2$, we can rewrite $s>s(d,\beta)$ as
$\beta \in \big[1, \frac{d(d-1)}{d^2-2(d+1)s}\big)$ when $\frac{d}{2(d+1)}<s<\frac d4$, while
$\beta \in \big[1,\frac{d}{d-2s}\big)$ when $\frac d4\le s<\frac d2$. If $\frac d4<s<\frac d2$, the range of $\beta$ is sharp up to the endpoint by Theorem \ref{thm:ex}.\footnotemark
\footnotetext{Although the upper bound $\beta \le \frac{d}{d-2s}$ in \cite{BKS_TLMS} is stated only for $q=2$, the same argument applies to all $q\ge2$.}
The estimates of \cite[Theorem 2.2]{BKS_TLMS} are stated with the $L_x^1L_t^\infty$ norm. Their proof, in particular \cite[(6.13) and (6.14)]{BKS_TLMS}, also gives the frequency-localized bounds with the $L_x^\beta L_t^\infty$ norm needed here. Interpolation with the single-function maximal estimates and summation over frequencies yield Theorem \ref{thm:max}; see Appendix \ref{sec:known} for details.

To complement the sufficient condition in Theorem \ref{thm:max}, for $1\le l \le d-1$ we combine a single-particle example in $l$ variables with the many-particle example of \cite{BKS_TLMS} in the remaining $d-l$ variables. The single-particle example is given by Bourgain for $l\ge2$ and by Dahlberg--Kenig for $l=1$. For $1\le \beta \le \infty$, set
\begin{align}\label{nec range}
\sigma(d,\beta):=\max_{l\in \{0,1,\dots,d\}}\Big\{ \frac{l}{2(l+1)}+\frac{d-l}2\big(1-\frac 1\beta\big)\Big\}
\end{align}
\begin{thm}\label{thm:ex}
Let $d\ge1$, $1\le \beta \le \infty$, and $s\in \R$. Suppose there exists a constant $C>0$ such that
\begin{align}\label{beta-beta}
\Big\| \sup_{t\in I} \Big|\sum_j \lambda_j \big|e^{it\Delta}f_j\big|^2\Big| \Big\|_{L_x^\beta(B_1^d)}
\leq C\|\lambda\|_{\ell^\beta}
\end{align}
holds for all orthonormal families $(f_j) \subset H^s(\R^d)$ and all sequences $(\lambda_j) \in \ell^\beta(\C)$. Then $s\ge \sigma(d,\beta)$.
\end{thm}

The extreme choices $l=d$ and $l=0$ in \eqref{nec range} recover, respectively, the single-particle lower bound $s\ge \frac{d}{2(d+1)}$ of Bourgain \cite{Bourgain16} (\cite{DK82} when $d=1$) and the many-particle lower bound $s\ge \frac d2(1-\frac1\beta)$ of Bez, Kinoshita and the third author \cite{BKS_TLMS}.

For $2\le\beta\le\infty$, one has $s(d,\beta)=\sigma(d,\beta)=\frac d2(1-\frac1\beta)$. Therefore, Theorems~\ref{thm:max} and \ref{thm:ex} give matching sufficient and necessary regularity bounds up to the endpoint. \eqref{beta-beta} holds for $s>\frac d2(1-\frac1\beta)$ and can hold only if $s\ge\frac d2(1-\frac1\beta)$.
For $d\ge2$, the intermediate choices $l\in\{1,\dots,d-1\}$ give a strictly stronger necessary condition than the two extreme choices precisely when
\[
\frac{d(d+1)}{d(d+1)-1}<\beta<2.
\]
For example, when $d=2$, we have
\begin{equation*}
\sigma(2,\beta)
=
\begin{cases}
\dfrac13,
& 1\leq\beta\leq\dfrac65,\\[2mm]
\dfrac34-\dfrac1{2\beta},
& \dfrac65\leq\beta\leq2,\\[2mm]
1-\dfrac1\beta,
& 2\leq\beta\leq\infty,
\end{cases}
\qquad
s(2,\beta)
=
\begin{cases}
\dfrac23-\dfrac1{3\beta},
&1\leq\beta\leq2,\\[2mm]
1-\dfrac1\beta,
&2\leq\beta\leq\infty.
\end{cases}
\end{equation*}
For comparison, Theorem \ref{thm:max} gives \eqref{beta-beta} whenever $s>s(2,\beta)$.
Figure \ref{fig:ex} compares the necessary and the  sufficient regularity conditions.
In dimension two, the improvement over the two previous necessary
bounds is strict for $6/5<\beta<2$. 
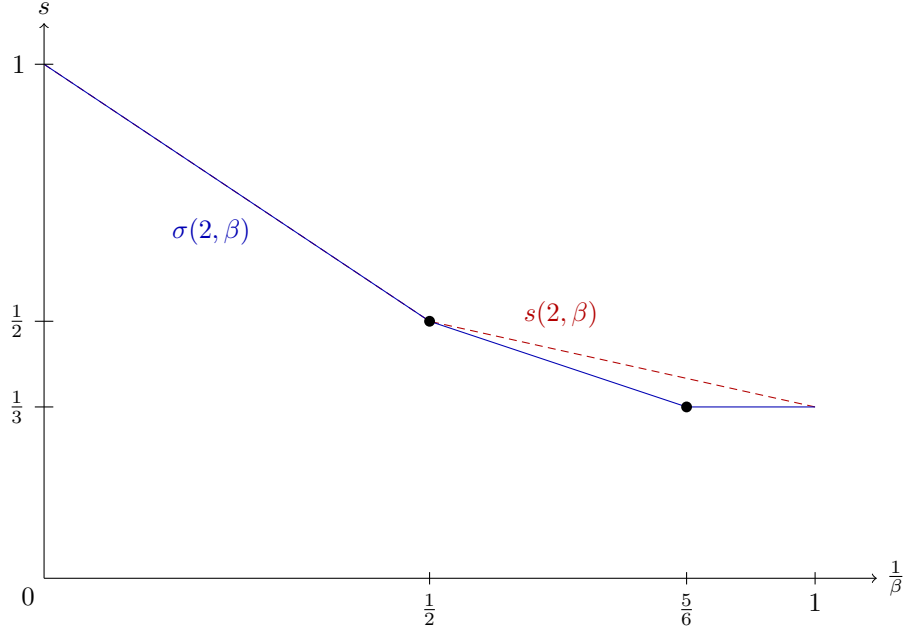
\begin{figure}[t]
\centering
\begin{tikzpicture}[x=6cm,y=4cm, scale=1.7]
\draw[->] (0,0) -- (1.08,0)
node[right] {$\frac 1\beta$};
\draw[->] (0,0) -- (0,1.08)
node[above] {$s$};
\draw[red!70!black,densely dashed]
(0,1)
-- ({1/2},{1/2})
-- (1,{1/3});
\draw[blue!70!black,]
(0,1)
-- ({1/2},{1/2})
-- ({5/6},{1/3})
-- (1,{1/3});
\fill ({1/2},{1/2}) circle (1.2pt);
\fill ({5/6},{1/3}) circle (1.2pt);
\node[below left] at (0,0) {$0$};
\draw ({1/2},0.012) -- ({1/2},-0.012)
node[below] {$\frac12$};
\draw ({5/6},0.012) -- ({5/6},-0.012)
node[below] {$\frac56$};
\draw (1,0.012) -- (1,-0.012)
node[below] {$1$};
\draw (0.012,{1/3}) -- (-0.012,{1/3})
node[left] {$\frac13$};
\draw (0.012,{1/2}) -- (-0.012,{1/2})
node[left] {$\frac12$};
\draw (0.012,1) -- (-0.012,1)
node[left] {$1$};
\node[blue!70!black,below left]
at (0.28,0.72)
{$\sigma(2,\beta)$};
\node[red!70!black,above right]
at (0.61,0.47)
{$s(2,\beta)$};
\end{tikzpicture}

\caption{
Known necessary and sufficient regularity bounds in dimension $d=2$. The solid curve represents the necessary condition $\sigma(2,\beta)$ from Theorem \ref{thm:ex}, while the dashed curve represents the known sufficient condition $s(2,\beta)$ from Theorem \ref{thm:max}. For $\beta\in[2,\infty]$, the two curves coincide. The sufficient condition is $s>s(2,\beta)$ and the necessary condition is $s\ge\sigma(2,\beta)$.
}
\label{fig:ex}
\end{figure}

We now extend \eqref{sup-schr} to orthonormal systems for the Schr\"odinger operators with quadratic potentials and the special Hermite operator.

\begin{thm}\label{thm:on-max-gen}
Let $\mathcal A$ be either $\mathcal H_\omega$ on $\R^d$ or $\mathcal L$ on $\R^{2d}$, and set $n=d$ in the former case and $n=2d$ in the latter. Let $\beta \in [1,\infty]$ and $s>s(n,\beta)$.
Fix $0<T_{\mathcal L}<\pi/2$. There exists $T_{\mathcal H_\omega}\in (0,1]$, depending on $\omega$, such that the following holds. Set
$I_{\mathcal A}=[0,T_{\mathcal A}]$.
Then the estimate
\begin{align}\label{S-op}
\Big\| \sup_{t\in I_{\mathcal A}} \Big| \sum_j \lambda_j |e^{-it\mathcal A} f_j|^2 \Big| \Big\|_{L_x^{\beta}(B_1^n)}
\lesssim \|\lambda\|_{\ell^\beta}
\end{align}
holds for all orthonormal systems $(f_j)\subset H^s(\R^n)$ and all sequences $\lambda =(\lambda_j)\in \ell^\beta(\C)$.
\end{thm}
By Theorem \ref{thm:equiv} below, the necessary condition $s\ge\sigma(n,\beta)$ from Theorem \ref{thm:ex} also applies to $\mathcal H_\omega$ and $\mathcal L$ for $\beta\in[1,2]$ on the respective intervals $I_{\mathcal A}$.
The proof of Theorem \ref{thm:on-max-gen} builds on the work of Cho--Lee--Vargas \cite{CLV}. A key ingredient is an equivalence principle relating the free Schr\"odinger evolution to a class of generalized Schr\"odinger propagators. To describe this class, we introduce the generalized Fourier integral operator
\[
S_{\mathfrak a}^\zeta f(x,t)
=(2\pi)^{-d}\int e^{i\zeta(x,t)\cdot \xi-i\sum_{j=1}^d \mathfrak a_j(t)\xi_j^2}\,\widehat f(\xi)\,d\xi,
\]
where $\mathfrak a=(\mathfrak a_1,\cdots, \mathfrak a_d)\in C^2([0,1])$ satisfies 
\begin{align}\label{deriv-a}
A_0:=\min_j\inf_{t\in [0,1]}\fa_j'(t) >0,
\end{align}
and where a map $\zeta_t:=\zeta(\cdot,t): \R^d \rightarrow \R^d$, $t\in [0,1]$, satisfies the bi-Lipschitz condition in $x$,
\begin{align}\label{lip-const}
D_0^{-1} |x-y|\le |\zeta_t(x)-\zeta_t(y)|\le D_0|x-y|
\end{align}
and the Lipschitz condition in $t$,
\begin{align}\label{zeta-holder}
|\zeta_t(x)-\zeta_{t'}(x)|\le D_0|t-t'|
\end{align}
for all $x,y\in B_1^d$ and $t,t'\in[0,1]$. Here $D_0\ge1$ is a constant. On the respective time interval $I_{\mathcal A}$,
both $e^{-it\mathcal H_\omega}$ and $e^{-it\mathcal L}$ can be represented as $S_\fa^\zeta f$, multiplied by a bounded function. See \cite[Lemma 4.1]{CLV} for $e^{-it\mathcal H_\omega}$ and Section \ref{single function} for $e^{-it\mathcal L}$.

\begin{rmk}\label{rmk:normalize time}
If $\mathfrak a$ and $\zeta$ satisfy the analogues of
\eqref{deriv-a}, \eqref{lip-const}, and \eqref{zeta-holder}
on $[0,T_{\mathcal A}]$, then the rescaled functions
$\widetilde{\mathfrak a}_j(t):=\mathfrak a_j(T_{\mathcal A}t)$,
$\widetilde{\zeta}(x,t):=\zeta(x,T_{\mathcal A}t)$
satisfy these conditions on $[0,1]$, possibly with different
constants $A_0$ and $D_0$.
\end{rmk}

In \cite{CLV}, the authors established an equivalence between mixed-norm estimates for $S_{\mathfrak a}^{\zeta}f$ and those for the free Schr\"odinger propagator $e^{it\Delta}f$ in the single-function setting. See Theorem~\ref{thm:CLV}, stated  in Section \ref{sub:twisted}. 
We prove an orthonormal analogue of the equivalence theorem of \cite{CLV}. Under the conditions below, the estimates for $S_{\fa}^\zeta$ hold for every $s>s_0$ if and only if the corresponding estimates for the free Schr\"odinger propagator hold.
Thus, the question of optimal regularity for orthonormal estimates for this class of generalized propagators reduces
to the corresponding question for the free Schr\"odinger operator. The maximal estimates for the Hermite and special Hermite operators follow as applications.

Our main result in this direction is the following.

\begin{thm}\label{thm:equiv}
Let $q,r\in[2,\infty]$, $\beta\in[1,2]$, and $s_0\in\R$, and assume that $2\beta\le r$. Set 
$Y(I)=L_x^{\frac q2}(B_1^d;L_t^{\frac r2}(I))$ or
$Y(I)=L_t^{\frac r2}(I;L_x^{\frac q2}(B_1^d))$.
In the case
$Y(I)=L_x^{q/2}(B_1^d;L_t^{r/2}(I))$, assume additionally that $2\beta\le q$.
Assume that $\fa$ and $\zeta$ satisfy \eqref{deriv-a}, \eqref{lip-const}, and \eqref{zeta-holder}.
Then the following are equivalent.
\begin{enumerate}[$(i)$]
\item 
For $s>s_0$, there exists a constant $C=C(s)>0$ such that
\begin{align}\label{on-sch-2}
\Big\| \sum_j \lambda_j |e^{it\Delta}f_j|^2 \Big\|_{Y(I) 
} \le C\|\lambda\|_{\ell^\beta}
\end{align}
holds for all orthonormal initial data $(f_j) \in H^s(\R^d)$ and all sequences
$\lambda=(\lambda_j)\in \ell^\beta$.
\item 
For $s>s_0$,  there exists a constant $C=C(s)>0$ such that
\begin{align}\label{on-hermite-2}
\Big\| \sum_j \lambda_j |S_{\fa}^\zeta  f_j|^2 \Big\|_{Y(I)
} \le C \|\lambda\|_{\ell^\beta}
\end{align}
holds for all orthonormal initial data $(f_j) \in H^{s}(\R^d)$
and all sequences
$\lambda=(\lambda_j)\in \ell^\beta$. 
\end{enumerate}
\end{thm}
To prove Theorem \ref{thm:equiv}, we adapt the time
localization argument of \cite{CLV} to the orthonormal setting.
At frequency $2^k$, we first compare the propagators on intervals
of length comparable to $2^{-k}$ and then combine these estimates
in the dual Schatten spaces to obtain estimates on a fixed interval.

Combined with the known free Schr\"odinger estimates, Theorem \ref{thm:equiv} yields Theorem \ref{thm:on-max-gen} for $1\le\beta\le2$. Interpolation with a direct bound at $\beta=\infty$ gives the remaining range. Thus the estimates for quadratic potentials and the special Hermite operator follow from the orthonormal time localization argument and the known free estimates.
To the best of our knowledge, these are the first maximal estimates in this setting.

Related orthonormal Strichartz estimates for the Hermite and special Hermite operators have been studied by Bez et al. \cite{BHLNS}, Mondal--Swain \cite{MS}, and Ghosh--Mondal--Swain \cite{GMS}. The estimates for the Hermite operator in \cite{BHLNS} use the Sj\"ogren--Torrea relation \cite{ST}, while the estimates in \cite{MS,GMS} were obtained through
bounds for the associated operators.
Further orthonormal Strichartz estimates for general flows associated with the special Hermite operator, including an endpoint obstruction, were obtained by Ghosh--Swain \cite{GS_25}.

As a further application, the identity in Proposition \ref{prop:sp-stri}, together with the Strichartz estimates of \cite{FLLS, FS_AJM, BHLNS}, gives a short proof of \cite[Theorem 1.6]{GMS}. See Corollary~\ref{cor:GMS}. Theorem \ref{thm:equiv} also gives local Strichartz estimates with Sobolev regularity, as discussed in Section~\ref{subsec:loc sch}.

For $d\ge2$ and $1<\beta<2$, a gap remains between the necessary and sufficient regularity bounds above. Theorem~\ref{thm:equiv} transfers any improvement of the free sufficient bound to $S_{\fa}^\zeta$, and hence to the Hermite and special Hermite operators.

\subsection{Pointwise convergence for orthonormal systems}
Let $\gamma_0$ be a self-adjoint operator on $L^2(\R^n)$. We set
\[
\gamma(t)=e^{-it\mathcal A} \gamma_0 e^{it\mathcal A}.
\]
Then $\gamma(t)$ solves the operator-valued equation
\begin{align*}
i\partial_t \gamma = [\mathcal A, \gamma], \quad \gamma(0)=\gamma_0.
\end{align*}
For $\mathcal A=(-\Delta)^{a/2}$, this is the fractional von Neumann--Schr\"odinger equation with free interaction. We now apply the maximal estimates to density functions of $\gamma(t)$ in a certain class.

For $1\le \beta\le \infty$, let $\mathcal C^\beta(L^2(\R^n))$ be the Schatten class of compact operators $T$ on $L^2(\R^n)$  with the norm $\|T\|_{\mathcal C^\beta}=\|\nu\|_{\ell^\beta}$, where $\nu=(\nu_j)$ are the eigenvalues of $|T|=\sqrt{TT^*}$. Here, $\mathcal C^\infty(L^2)=\mathcal B(L^2)$ denotes the bounded operators equipped with the operator norm. We also define Sobolev type Schatten spaces
\[
\mathcal C^{\beta,s}
=\big\{ \gamma \in \text{Com}(H^{-s}(\R^n),H^s(\R^n)):
\big\| \langle D\rangle^{s} \gamma \langle D \rangle^{s} \big\|_{\mathcal C^\beta(L^2(\R^n))}<\infty \big\}.
\]
Here, $\operatorname{Com}(H^{-s},H^s)$ denotes the space of compact operators from $H^{-s}(\R^n)$ to $H^s(\R^n)$, and $\langle D \rangle = (1-\Delta)^{1/2}$ is the Fourier multiplier with symbol $\langle \xi \rangle = (1+|\xi|^2)^{1/2}$.

Let $\gamma_0\in\mathcal C^{\beta,s}$ be self-adjoint, and set
$\Gamma_0=\langle D\rangle^s\gamma_0\langle D\rangle^s$.
Let $(\mu_j,g_j)$ be the nonzero eigenvalues and eigenvectors of $\Gamma_0$,
with $(g_j)$ orthonormal in $L^2(\R^n)$.
Then $h_j=\langle D\rangle^{-s}g_j$ form an orthonormal
system in $H^s(\R^n)$, and
$
\gamma_0g=\sum_j\mu_j\langle g,h_j\rangle h_j,
$
and $\|(\mu_j)\|_{\ell^\beta}
=\|\gamma_0\|_{\mathcal C^{\beta,s}}.
$
For finite-rank $\gamma_0$, restricting the kernels to the
diagonal gives
$
\rho_{\gamma_0}(x)=\sum_j\mu_j|h_j(x)|^2
$
and
\begin{align}\label{density sum}
\rho_{\gamma(t)}(x)
=\sum_j\mu_j|e^{-it\mathcal A}h_j(x)|^2.
\end{align}
For general $\gamma_0$, finite-rank approximation defines
$\rho_{\gamma_0}$ in $L^1_{\mathrm{loc}}(\R^n)$ whenever
$s>\frac n2(1-\frac1\beta)$; see
\cite{BLN_Forum,BLN_Selecta,BKS_TLMS,KKS}.
In the range $s>s(n,\beta)$ considered below, the maximal
estimates also define the evolved density.

We study the pointwise convergence
\begin{equation}\label{e:pw density}
\lim_{t\to0+}\rho_{\gamma(t)}(x)
=
\rho_{\gamma_0}(x),\quad \text{a.e.} \,\,x\in \R^n
\end{equation}
for all self-adjoint $\gamma_0\in\mathcal C^{\beta,s}$. This is an analogue of Carleson's problem for infinitely many particles.
Theorem \ref{thm:on-max-gen} gives the following result.

\begin{cor}
Let $1\le \beta \le \infty$, and let $\mathcal A$ be either $\mathcal H_\omega$ on $\R^d$ or $\mathcal L$ on $\R^{2d}$.
Set $n=d$ in the former case and $n=2d$ in the latter. Then
\eqref{e:pw density} holds for all self-adjoint $\gamma_0 \in \mathcal C^{\beta,s}$ provided $s>s(n,\beta)$.
\end{cor}

Theorems \ref{thm:max} and \ref{thm:equiv} also give convergence along curves for the Schr\"odinger operator.

\begin{cor}
Let $1\le \beta \le \infty$ and $s>s(d,\beta)$. Suppose $\zeta(x,0)=x$ for all $x\in \R^d$, and $\zeta$ satisfies \eqref{lip-const} and \eqref{zeta-holder} for $x,y\in B_1^d$ and $t,t'\in [0,1]$. 
Let $\gamma_0\in \mathcal C^{\beta,s}$ be a self-adjoint operator
and $\gamma(t)=e^{it\Delta}\gamma_0 e^{-it\Delta}$. Then
\[
\lim_{t\rightarrow 0+}\rho_{\gamma(t)}(\zeta(x,t))= \rho_{\gamma_0}(x), \quad \text{a.e.} \quad~ x\in \R^d.
\]
\end{cor}
Using the decomposition above with $n=d$ and $\fa(t)=(t,\dots,t)$, we have
\[
\rho_{\gamma(t)}(\zeta(x,t))
=\sum_j\mu_j
  |e^{it\Delta}h_j(\zeta(x,t))|^2
=\sum_j\mu_j
  |S^\zeta_{\fa}h_j(x,t)|^2.
\]
Theorem \ref{thm:max} and \ref{thm:equiv}, with $(q,r)=(2\beta,\infty)$ give the required maximal estimate for each fixed ball. (See Subsection \ref{sec:proof:on-gen} below). 
Both corollaries follow from the standard argument of deducing pointwise convergence from the maximal estimate, we omit the details and refer interested readers to \cite{BKS_TLMS, BLN_Selecta} for further reading.

Pointwise convergence along curves has been extensively studied in the single-particle setting. 
Beyond the works of \cite{LR,CLV}  there has been substantial further development in dimension one. We refer interested readers to \cite{ChoShi} for a more comprehensive overview.
In higher dimensions, the works \cite{CM, Ming, MSV} establish tangential convergence along $\alpha$-H\"older curves.

The tangential case for orthonormal systems, such as $\zeta(x,t)=x-t^\kappa$, $0<\kappa<1$, in $\R$ would require an orthonormal analogue of the tangential maximal estimates of \cite{CLV}, and is not a direct consequence of the equivalence principle proved here. This case is treated in \cite{JK26} in one dimension, for fractional Schr\"odinger operators along H\"older curves.

\subsection*{Organization of the paper}
The paper is organized as follows. In Section \ref{sub:twisted}, we establish the pointwise identity for the special Hermite propagator and prove Theorems \ref{thm:twisted-single} and \ref{thm:on-max-gen}. Section \ref{sec:4} proves the equivalence on superlocal time intervals and reduces Theorem \ref{thm:equiv} to the time localization result proved in Section \ref{sec:5}. In Section~\ref{sec:example}, we combine the single-particle and many-particle examples to prove the necessary condition $s\ge\sigma(d,\beta)$ in Theorem \ref{thm:ex}. Section \ref{sec:7} derives Strichartz estimates from the pointwise identity and the equivalence theorem. Finally, Appendix \ref{sec:known} verifies the $L_x^\beta$ form of the known free estimate stated in Theorem \ref{thm:max}.

\subsection*{Notation}  
Throughout the paper $A\lesssim B$ denotes $A\le CB$ for some constant independent of the relevant parameters.

\section{Maximal estimates: Proofs of Theorem~\ref{thm:twisted-single} and Theorem ~\ref{thm:on-max-gen}}\label{sub:twisted}

We first prove the single-function maximal estimates for $e^{-it\mathcal L}$ in Theorem \ref{thm:twisted-single}. 
We then prove the orthonormal maximal estimates in Theorem \ref{thm:on-max-gen} assuming the equivalence principle in Theorem \ref{thm:equiv}.

\subsection{Proof of Theorem \ref{thm:twisted-single}}\label{single function}
We use the following equivalence principle of Cho--Lee--Vargas \cite{CLV}.
\begin{thm}[{\cite[Theorem 1.6]{CLV}}]\label{thm:CLV} 
Let $n\ge1$, $q,r\in [2,\infty]$ and $s_0\in \R$. Suppose that $\zeta$ and $\fa$ satisfy  \eqref{deriv-a}, \eqref{lip-const}, and \eqref{zeta-holder}, with $d$ replaced by $n$. Then 
\[ \|e^{it \Delta} f\|_{L^q_z(B_1^n; L^r_t(I))}\le C \|f\|_{H^s(\R^n)}\]
    holds for all $s> s_0$ if and only if
\[ \|S^\zeta_{\fa} f\|_{L^q_z(B_1^n;L^r_t(I))}\le C \|f\|_{H^s(\R^n)}\]
holds for all $s>s_0$.    
\end{thm}
\begin{proof}[Proof of Theorem~\ref{thm:twisted-single}]
Let $n=2d$. Fix $0<T_{\mathcal L}<\pi/2$ and set $I_{\mathcal L}=[0, T_{\mathcal L}]$. For $f\in\mathcal{S}(\R^{n})$ and $t\in[0,\frac\pi2)$, we claim that
    \begin{equation}\label{id-special}
    e^{-it\mathcal L}f(z)
=
(1+\tan^2 t)^{\frac d2} e^{-i(\tan t)\frac{|z|^2}4}e^{i(\tan t)\Delta}f\big(z-(\tan t)Jz\big),
    \end{equation}
    where 
    \[
J=\begin{pmatrix} 0 & -I_d \\ I_d & 0\end{pmatrix}.
\]
Here $I_d$ is the identity matrix on $\R^d$.
Set $\fa(t)=(\tan t,\cdots,\tan t)$ and $\zeta_t(z)= z - (\tan t)Jz$. 
The identity \eqref{id-special} gives
\begin{align*}
\big\| e^{-it\mathcal L}f \big\|_{L_z^q(B_1^n;L_t^r(I_{\mathcal L}))}
\lesssim
\big\| S_{\fa}^\zeta f \big\|_{L_z^q(B^n_1; L_t^r(I_{\mathcal L}))}.
\end{align*}

Assuming \eqref{id-special} for the moment, we verify \eqref{deriv-a}, \eqref{lip-const}, and \eqref{zeta-holder}. 
The condition \eqref{deriv-a} holds since $\fa \in C^2(I_{\mathcal L})$ satisfies $\fa_j'(t)=\sec^2t \geq 1$ on $I_{\mathcal L}$. 
Next, since $J^{\top} J=I_{2d}$, we have $|\zeta_t(z)-\zeta_t(w)|=(1+\tan^2t)^{\frac12} |z-w|$, which proves the bi-Lipschitz condition \eqref{lip-const}. 
Finally, for $t,t'\in I_{\mathcal L}$ and $z\in B_1^n$,
\[
|\zeta_t(z)-\zeta_{t'}(z)|=|\tan t-\tan t'||Jz|
\leq 
\sec^2(T_{\mathcal L})|t-t'|
\]
by the mean value theorem. 
By Remark \ref{rmk:normalize time}, we may rescale $[0,T_\mathcal L]$ to $I=[0,1]$ and apply Theorem \ref{thm:CLV}. 
Taking $(q,r)=(2,\infty)$ and $s_0=\frac{n}{2(n+1)}$, Theorems \ref{thm:CLV} and \ref{t:max schrodinger single} give the desired bound.

It remains to verify \eqref{id-special}. Let $0<t<\pi/2$.
We recall the explicit formulas for the free Schr\"odinger propagator
$e^{it\Delta}$ and the special Hermite propagator $e^{-it\mathcal L}$ on $\R^{2d}$.
For the Schr\"odinger propagator, we have (see \cite[p. 1370]{LR})
\begin{equation}\label{classic2}
e^{it\Delta}f(z) = e^{-i \frac {\pi d}2} (4\pi t)^{-d} \int_{\R^{2d}} e^{i \frac{|z-z'|^2}{4t}} f(z') dz'.
\end{equation}
On the other hand, Mehler's formula gives
\begin{equation}\label{tw}
e^{-it\mathcal L}f(z)
=
e^{-i\frac{\pi d}2}(4\pi\sin t)^{-d}
\int_{\mathbb{R}^{2d}}
e^{\,i(\frac{|z-z'|^2}{4}\cot t-\frac12 z^\top Jz')}
f(z')\,dz'.
\end{equation}
See \cite[p. 21]{Thangavelu2} and also \cite{Thangavelu, ST, JLR}.
Set $\tau=\tan t$. Since $\sin^2t=\tau^2/(1+\tau^2)$, \eqref{tw} becomes
\begin{equation}\label{e:sht}
e^{-i  t \mathcal L} f(z) 
=e^{-i \frac{\pi d}2} (1+\tau^2)^{\frac d2}(4\pi \tau)^{-d} \int_{\R^{2d}} e^{i (\frac{|z-z'|^2}{4\tau} -\frac12 z^{\top} J z' )} f(z') dz'.
\end{equation}
We now rewrite the phase function.
Since $\langle z,J z\rangle=0$ and $J^\top J=I_{2d}$, we have
\[\frac{|z-z'|^2}{4\tau}-\frac12z^\top J z' =\frac{|z-\tau J z-z'|^2}{4\tau} -\frac{\tau|z|^2}4.\]
Substituting this identity into \eqref{e:sht} and using \eqref{classic2}, we obtain
\begin{align*}
e^{-i t \mathcal L} f(z) 
&=  (1+\tau^2)^{\frac d2} e^{-i\frac{\tau|z|^2}4} e^{-i \frac{\pi d}2} (4\pi \tau)^{-d}\int_{\R^{2d}} e^{i\frac{|z-\tau J z-z'|^2}{4\tau}} f(z') dz' \\
&=  (1+\tau^2)^{\frac d2} e^{-i\frac{\tau|z|^2}4}  e^{i\tau \Delta} f(z- \tau J z).
\end{align*}
Since $\tau=\tan t$, this is precisely \eqref{id-special}.
This completes the proof.
\end{proof}

\subsection{Duality}

We record the duality principle of Frank--Sabin \cite{FS_AJM}, in the mixed norm form used below. See also \cite{BHLNS}. It relates orthonormal estimates to Schatten bounds.

\begin{lem}[Duality]\label{lem:dual}
Suppose $T$ is a bounded operator from $L^2$ to $L_x^qL_t^r$ for some $q,r\ge2$. For $\beta\ge1$, the following are equivalent.
\begin{enumerate}[(i)]
\item
For all orthonormal initial data $(f_j)\subset L^2$ and $\lambda=(\lambda_j)\in \ell^\beta$,
\[
\Big\| \sum_j \lambda_j |T f_j|^2 \Big\|_{L_x^{\frac q2}L_t^{\frac r2}} \le C \|\lambda\|_{\ell^\beta}.
\]
\item
For any $W\in L_x^{\tilde q}L_t^{\tilde r}$,
\begin{align*}
\big\| WTT^*\overline W\|_{\mathcal C^{\beta'}}
\le C' \|W\|^2_{L_x^{\tilde q}L_t^{\tilde r}}.
\end{align*}
Here, $\beta'$ denotes the H\"older conjugate of $\beta$, and $\tilde q$ and $\tilde r$ satisfy
\begin{align}\label{tilde-qr}
\frac 1{\tilde q}+\frac 1q=\frac 12\quad\text{and}\quad \frac 1{\tilde r}+\frac 1r=\frac 12.
\end{align}
\end{enumerate}
Moreover, the optimal constants $C$ and $C'$ coincide. In addition, the same remains valid if the mixed norm spaces $L^{\frac q2}_xL^{\frac r2}_t$ and $L^{\tilde q}_x L^{\tilde r}_t$ are replaced by $L^{\frac r2}_tL^{\frac q2}_x$ and $L^{\tilde r}_tL^{\tilde q}_x$, respectively.
\end{lem} 

When $\beta=1$, the dual estimate is understood as an operator norm bound, without a compactness requirement.
We write $\mathcal B(L^2(\R^d))$ for the space of bounded operators on $L^2(\R^d)$. We also record the following consequence of Lemma \ref{lem:dual}, which will be used repeatedly.

\begin{lem}\label{lem:TB-T1}
Let $q,r\in[2,\infty]$ and $\beta\in[1,\infty]$.
Let $Y=L_x^{q/2}L_t^{r/2}$ or $Y=L_t^{r/2}L_x^{q/2}$, and suppose that $T$ is a bounded operator from $L^2(\R^d)$ to $L_x^qL_t^r$ or $L_t^rL_x^q$, respectively.
Let $Q\in\mathcal B(L^2(\R^d))$. If
\begin{align}\label{assume-T}
\big\|\sum_j \lambda_j|Tf_j|^2 \big\|_{Y} \le C\|\lambda\|_{\ell^\beta}
\end{align}
holds for all orthonormal systems $(f_j) \subset L^2(\R^d)$ and all $\lambda=(\lambda_j)\in \ell^\beta$, then
\begin{equation}\label{result-TQ}
\big\|\sum_j \lambda_j|TQf_j|^2 \big\|_{Y} \le C\|Q\|_{\mathcal B(L^2)}^2\|\lambda\|_{\ell^\beta}
\end{equation}
holds for all orthonormal systems $(f_j) \subset L^2(\R^d)$ and all $\lambda=(\lambda_j)\in \ell^\beta$.
\end{lem}

\begin{proof}
Let $\widetilde Y$ be $L_x^{\tilde q}L_t^{\tilde r}$ when $Y=L_x^{q/2}L_t^{r/2}$ and $\widetilde Y=L_t^{\tilde r}L_x^{\tilde q}$ when $Y=L_t^{r/2}L_x^{q/2}$  with $\tilde q, \tilde r$ defined in \eqref{tilde-qr}. 
By Lemma \ref{lem:dual}, \eqref{assume-T} and \eqref{result-TQ} are equivalent to
\begin{align}\label{assume-WT}
\|WTT^*\overline{W}\|_{\mathcal C^{\beta'}}\le C\|W\|_{\widetilde Y}^2
\end{align}
and
\begin{equation}\label{result-WTQ}
\|WTQ(TQ)^*\overline{W}\|_{\mathcal C^{\beta'}}\le C\|Q\|_{\mathcal B(L^2)}^2\|W\|_{\widetilde Y}^2, 
\end{equation}
respectively.
Therefore, it suffices to show \eqref{result-WTQ} assuming \eqref{assume-WT}.

By H\"older's inequality in the Schatten spaces (\cite[Theorem 2.7]{Simon}), we obtain
\begin{align*}
\|WTQ(TQ)^*\overline{W}\|_{\mathcal C^{\beta'}}=\|WTQ\|_{\mathcal C^{2\beta'}}^2
\le 
\|Q\|_{\mathcal{B}(L^2)}^2\|WT\|_{\mathcal C^{2\beta'}}^2
=\|Q\|_{\mathcal{B}(L^2)}^2\|WTT^*\overline{W}\|_{\mathcal C^{\beta'}}.
\end{align*}
Applying \eqref{assume-WT} proves \eqref{result-WTQ}.
\end{proof}

\subsection{Proof of Theorem \ref{thm:on-max-gen}}\label{sec:proof:on-gen}
We now prove Theorem \ref{thm:on-max-gen} by combining the equivalence principle in Theorem \ref{thm:equiv} with  the orthonormal maximal estimates for the Schr\"odinger operator in Theorem \ref{thm:max}.

As discussed in \cite[Lemma 4.1]{CLV} and the subsection~\ref{single function},
both $e^{-it\mathcal H_\omega}$ and $e^{-it\mathcal L}$ can be represented,
locally in time, in terms of $S_{\fa}^\zeta$ for suitable choices of
$\fa$ and $\zeta$.
Hence, it suffices to prove the desired estimates \eqref{S-op} for $S_{\fa}^\zeta$ instead of $e^{-it\mathcal A}$.

Let $1\le\beta\le2$. Applying Theorem \ref{thm:max} in dimension $n$ and Theorem \ref{thm:equiv} with $(q,r)=(2\beta,\infty)$,  we obtain
\begin{equation}\label{betabeta}
\Big\| \sup_{t\in I} \Big| \sum_j \lambda_j \big| S_{\fa}^\zeta \langle D\rangle^{-s} g_j \big|^2 \Big| \Big\|_{L_x^{\beta}(B_1^n)}
\lesssim \|\lambda\|_{\ell^\beta}
\end{equation}
for every $s>s(n,\beta)$ (cf. \eqref{s0nbeta}), every orthonormal system $(g_j)\subset L^2(\R^n)$,
and every $(\lambda_j)\in \ell^\beta$.

Next, consider $\beta=\infty$. For every $(x,t)\in B_1^n\times I_{\mathcal A}$, we
define $\overline{h_{x,t}}(\xi)=e^{i\zeta_t(x)\cdot \xi-i\sum_{l=1}^n \fa_l(t)\xi_l^2}\langle \xi \rangle^{-s}$.
Then $S_{\fa}^\zeta \langle D\rangle^{-s} g_j= (2\pi)^{-n}\langle \,\widehat{g_j}\,, h_{x,t}\rangle$. For any orthonormal system $(g_j) \subset L^2(\R^n)$, $((2\pi)^{-n/2}\widehat{g_j})$ is also an orthonormal system in $L^2(\R^n)$ by Plancherel's theorem. Hence,
Bessel's inequality (cf. \cite{BHLNS}) gives, whenever $s>n/2$,
\begin{align*}
\sum_j|S_{\fa}^\zeta \langle D\rangle^{-s} g_j(x,t)|^2
\lesssim
\|h_{x,t}\|_2^2
\lesssim \|\langle \xi \rangle^{-s}\|_2^2 <\infty.
\end{align*}
Moreover, for every $(\lambda_j)\in\ell^\infty$,
\begin{align*}
\Big|\sum_j\lambda_j|S_{\fa}^\zeta\langle D\rangle^{-s}g_j(x,t)|^2\Big|
\le \|\lambda\|_{\ell^\infty}
\sum_j|S_{\fa}^\zeta\langle D\rangle^{-s}g_j(x,t)|^2
\lesssim \|\lambda\|_{\ell^\infty}.
\end{align*}
Consequently, \eqref{betabeta} holds with $\beta=\infty$ and $s>\frac n2$.

Finally, let $2< \beta<\infty$. Fix $\varepsilon>0$. To interpolate between the estimates \eqref{betabeta} at $\beta=2$ and $\beta=\infty$, set $s(2)=\frac n4+\varepsilon$ and $s(\infty)=\frac n2+\varepsilon$, and consider the analytic family $T_z = S_\fa^\zeta \langle D \rangle^{-s(z)}$
where $s(z)=(1-z)s(2)+zs(\infty)$ for $0\le \Re(z)\le 1$.
By Lemma \ref{lem:dual}, \eqref{betabeta} at $\beta=2$ and $\beta=\infty$ yield, respectively,
\begin{align}\label{dual interpol}
\|WT_{i\tau}\|_{\mathcal C^4}\lesssim\|W\|_{L_x^4L_t^2},
\qquad
\|WT_{1+i\tau}\|_{\mathcal C^2}\lesssim\|W\|_{L_x^2L_t^2},
\end{align}
uniformly in $\tau\in\mathbb R$.\footnotemark 
\footnotetext{Indeed, note that $T_{i\tau}=T_0Q_\tau$ and $T_{1+i\tau}=T_1Q_\tau$ for $Q_\tau=\langle D \rangle^{-i\tau(s(\infty)-s(2))}$ where $T_0=S_{\fa}^\zeta \langle D\rangle^{-s(2)}$ and $T_1=S_{\fa}^\zeta \langle D\rangle^{-s(\infty)}$.
Since $Q_\tau Q_\tau^*=I$ on $L^2(\R^n)$, we have
$
(WT_{i\tau})(WT_{i\tau})^*
=WT_0T_0^*\overline W,
$
and the estimate \eqref{betabeta} with $\beta=2$ and duality gives the estimate for $WT_{i\tau}$ in \eqref{dual interpol}. A similar computation applies to $WT_{1+i\tau}$.
}
Applying complex interpolation to the analytic family $W\rightarrow WT_z$ gives
\[
\|WT_\theta\|_{\mathcal C^{2\beta'}}
\lesssim\|W\|_{L_x^{2\beta'}L_t^2}
\]
where $\theta=1-\frac 2\beta$ and $\frac{1}{2\beta'}=\frac{1-\theta}4+\frac{\theta}2$.
Here, we use $[L_x^4L_t^2,L_x^2L_t^2]_\theta
=L_x^{2\beta'}L_t^2$ and $[\mathcal C^4,\mathcal C^2]_\theta
=\mathcal C^{2\beta'}$ (see Theorem 2.9 and Remark 1 in p.23 in \cite{Simon}).
Since $(1-\theta)s(2)+\theta s(\infty)=\frac n2(1-\frac1\beta)+\varepsilon$, another application of Lemma~\ref{lem:dual} proves \eqref{betabeta} for every $s>\frac n2(1-\frac1\beta)=s(n,\beta)$ and $\beta\in [2,\infty]$. Combining this with the case $\beta \in [1,2]$ proved above completes the proof.
 \qed

\section{Proof of Theorem \ref{thm:equiv}}\label{sec:4}
\subsection{Reductions to frequency localized estimates}
We now reduce Theorem \ref{thm:equiv} to an equivalence at each frequency-localized scale. 
Choose $\phi_0\in C_0^\infty(B^d_2)$ equal to $1$ on $B_1^d$, and $\phi(\xi)=\phi_0(\xi)-\phi_0(2\xi)$.
Then $\phi \in C^\infty_0(B^d_2\setminus B^d_{1/2})$ and $\phi_0(\xi)+\sum_{k\ge1}\phi(2^{-k}\xi)=1$ for $\xi \neq0$. Define the projection operator $P_0$ and $P_k$, $k\ge1$, by $\widehat{P_0f}=\phi_0\widehat f$ and $\widehat{P_kf}=\phi(2^{-k}\cdot)\widehat f$. 

\begin{prop}\label{prop:reduction}
Let $s\in\mathbb R$, $q,r\in[2,\infty]$, and $\beta\in[1,2]$, and
assume that $2\beta\le r$. Let $\fa$ and $\zeta$ satisfy
\eqref{deriv-a}, \eqref{lip-const}, and \eqref{zeta-holder}, and set
$Y(I)=L_x^{\frac q2}(B_1^d;L_t^{\frac r2}(I))$ or $Y(I)=L_t^{\frac r2}(I;L_x^{\frac q2}(B_1^d))$.
In the case $Y(I)=L^{\frac q2}(B_1^d;L_t^{\frac r2}(I))$, assume additionally that $2\beta\le q$.
Then the following estimates are equivalent, with constants independent of $k\ge1$:
\begin{enumerate}[$(i)$]
\item 
There exists a constant $C>0$ such that
\begin{align}\label{aaa}
\Big\|\sum_j\lambda_j|e^{it\Delta}P_kf_j|^2\Big\|_{Y(I)}
&\le C2^{2ks}\|\lambda\|_{\ell^\beta}
\end{align}
holds for all orthonormal initial data $(f_j)\subset L^2(\mathbb R^d)$ and all sequences $\lambda=(\lambda_j)\in \ell^\beta$.
\item 
There exists a constant $C>0$ such that
\begin{align}\label{bbb}
\Big\|\sum_j\lambda_j|S_{\fa}^\zeta P_kf_j|^2\Big\|_{Y(I)}
&\le C2^{2ks}\|\lambda\|_{\ell^\beta}
\end{align}
holds for all orthonormal initial data $(f_j)\subset L^2(\mathbb R^d)$ and all sequences $\lambda=(\lambda_j)\in \ell^\beta$. 
\end{enumerate}
\end{prop}

\noindent \textit{Proof of Theorem \ref{thm:equiv}.}
We first prove Theorem \ref{thm:equiv} assuming Proposition \ref{prop:reduction}.
Suppose first that \eqref{on-sch-2} holds for every Sobolev exponent larger than $s_0$. 
Fix $\tau>s_0$ and choose $s$ so that $s_0<s<\tau$.

Apply Lemma \ref{lem:TB-T1} to $T=e^{it\Delta} \langle D \rangle^{-s}$ with $Q=2^{-ks}\langle D \rangle^s P_k$.
Since $\|Q\|_{\mathcal B(L^2)}\lesssim 1$ and $e^{it\Delta}P_k=2^{ks}TQ$, we obtain \eqref{aaa}. Proposition \ref{prop:reduction} then gives \eqref{bbb} with the same exponent.

Let $(f_j)$ be an orthonormal system in $H^\tau(\R^d)$ and set $g_j=\langle D\rangle^\tau f_j$ so that $(g_j)$ are orthonormal in $L^2(\R^d)$. For $k\ge1$, choose $\widetilde P_k$ with symbol supported near $|\xi|\sim 2^k$ such that $P_k\widetilde P_k = P_k$. Set $Q_k=2^{\tau k}\langle D\rangle^{-\tau}\widetilde P_k$. Then $\|Q_k\|_{\mathcal B(L^2)}\lesssim 1$ uniformly in $k$, and $P_k \langle D\rangle^{-\tau}=2^{-\tau k}P_kQ_k$. Applying Lemma \ref{lem:TB-T1} to \eqref{bbb}  gives
\[
\Big\|\sum_j \lambda_j |S^\zeta_{\fa} P_{k}\langle D \rangle^{-\tau} g_j|^2\Big\|_{Y(I)}
\lesssim C
2^{2k(s-\tau)} \|\lambda\|_{\ell^\beta}.
\]
The low frequency term $P_0$ is handled directly by the Bessel's inequality, since $\phi_0(\xi)\langle \xi\rangle^{-\tau} \in L^2(\R^d)$ and $\|\lambda\|_{\ell^\infty} \le \|\lambda\|_{\ell^\beta}$.
By Minkowski's inequality for the sum over $k\ge0$, we have
\begin{align}\label{eq:Minkowski}
\Big\|\sum_j \lambda_j |S^\zeta_{\fa}f_j|^2\Big\|_{Y(I)}^{1/2}
&\le \sum_{k\ge0}\Big\|\sum_j |\lambda_j| |S^\zeta_{\fa}P_k f_j|^2\Big\|_{Y(I)}^{1/2} \\
&\lesssim \sum_{k\ge0}2^{k(s-\tau)}\|\lambda\|_{\ell^\beta}^{1/2}.
\end{align}
Since $s<\tau$, this yields \eqref{on-hermite-2} and hence $(i)\Rightarrow(ii)$.

The reverse implication follows by the same argument with  \eqref{aaa} and \eqref{bbb} interchanged. Thus, Proposition \ref{prop:reduction} implies Theorem \ref{thm:equiv}.
\qed

\subsection{Superlocal equivalence and globalization}
To prove Proposition \ref{prop:reduction}, and hence Theorem~\ref{thm:equiv}, we adapt the argument of \cite{CLV} for the single-function estimates to the orthonormal setting. 
Fixing a dyadic frequency scale $2^k$, we call estimates on $I=[0,1]$ local-in-time estimates
and those on subintervals $J\subset I$ with $|J|\sim 2^{-k}$ superlocal-in-time estimates.

The proof consists of two main steps. First, we establish an orthonormal analogue
of the single-function equivalence in \cite{LR,CLV} on each superlocal time
interval (see Proposition~\ref{lem:local equiv}). On each such interval, we linearize $\fa$ in time and freeze $\zeta_t$ at the center of the interval. A Fourier series expansion of the error factor then gives, after a change of variables, an expression in terms of the Schr\"odinger operator.

We then deduce the local-in-time estimate on $I$ from the superlocal estimates for $S_{\fa}^\zeta$ (see Proposition~\ref{lem:glob3}). In the dual Schatten formulation, we
decompose the associated $TT^*$ operator into diagonal, near-diagonal, and
far-off-diagonal terms in time. The diagonal and near-diagonal terms are
bounded using the superlocal estimate, orthogonality, and the finite overlap.
For the far-off-diagonal term, the kernel estimate in
Lemma \ref{lem:ker} gives a Hilbert--Schmidt bound.
The inclusion $\mathcal C^2\subset\mathcal C^{\beta'}$
for $\beta'\ge2$ then yields the required estimate
for $1\le\beta\le2$.

We first state the equivalence on superlocal time intervals.

\begin{prop}\label{lem:local equiv}
Let $s\in\R$, $\beta\in[1,2]$, and $q,r\in [2,\infty]$.
For $k\ge1$, let $J_0\subset I$ be an interval of length comparable to $2^{-k}$. Assume that $\fa$ and $\zeta$ satisfy \eqref{deriv-a}, \eqref{lip-const}, and \eqref{zeta-holder}. Set
$
Y(J_0)=L_x^{\frac q2}(B_1^d;L_t^{\frac r2}(J_0))$, or
 $Y(J_0)=L_t^{\frac r2}(J_0;L_x^{\frac q2}( B_1^d))$.
Then the following are equivalent.

\begin{itemize}
    \item [(i)] There exists a constant $C>0$ such that
    \begin{align}\label{sch-local}
\Big\| \sum_j \lambda_j |e^{it\Delta} P_{k}f_j|^2 \Big\|_{Y(J_0)} \le C 2^{2ks}\|\lambda\|_{\ell^\beta}
\end{align}
holds for all orthonormal system $(f_j) \subset L^2(\R^d)$ and all sequences $(\lambda_j)\in \ell^\beta$.
\item[(ii)] There exists a constant $C>0$ such that
\begin{align}\label{aux-local}
\Big\| \sum_j \lambda_j |S^\zeta_{\fa} P_{k} f_j|^2 \Big\|_{Y(J_0)} \le C2^{2ks}\|\lambda\|_{\ell^\beta}
\end{align}
holds for all orthonormal system $(f_j) \subset L^2(\R^d)$ and all sequences $(\lambda_j)\in \ell^\beta$.
\end{itemize}
\end{prop}

We next turn to the second step, in which we recover the local-in-time estimate on $I$ from the superlocal-in-time estimates. This is provided by the following proposition.

\begin{prop}\label{lem:glob3}
Let $k\ge0$, $s\in \R$, $q,r\in [2,\infty]$ and $\beta\in [1,2]$, and assume that $2\beta\le r$.
Suppose that $\fa$ and $\zeta$ satisfy \eqref{deriv-a}, \eqref{lip-const}, and \eqref{zeta-holder}.
For an interval $E\subset I$, set $Y(E)$ to be either
$Y(E)=L_x^{\frac q2}(B_1^d;L_t^{\frac r2}(E))$
or $Y(E)=L_t^{\frac r2}(E;L_x^{\frac q2}(B_1^d))$.
In the case $Y(E)=L_x^{q/2}(B_1^d;L_t^{r/2}(E))$, assume additionally that $2\beta\le q$.
Suppose that there exists a constant $C>0$, independent of $k$ and $J$, such that for every interval $J\subset I$ with $|J|\sim 2^{-k}$,
    \begin{equation}\label{e:lolo1}
        \Big\|\sum_j \lambda_j |S^\zeta_{\fa}P_{k} f_j|^2\Big\|_{Y(J)}\le C 2^{2ks} \|\lambda\|_{\ell^\beta} 
    \end{equation}
    for all orthonormal basis $(f_j)_j\subset L^2$ and $(\lambda_j)_j \in \ell^\beta$. Then there exists a constant $C'>0$ such that
    \begin{equation}\label{e:glgl1}
        \Big\|\sum_j \lambda_j |S^\zeta_{\fa}P_{k} f_j|^2\Big\|_{Y(I)}
        \lesssim C'
        2^{2ks} \|\lambda\|_{\ell^\beta} 
    \end{equation}
holds for all orthonormal basis $(f_j)_j \subset L^2$ and all $\lambda = (\lambda_j)_j\in\ell^\beta$. Here, $C'$ is independent of $k$.
\end{prop}

We now deduce Proposition \ref{prop:reduction} from Propositions \ref{lem:local equiv} and \ref{lem:glob3}.

\noindent\textit{Proof of Proposition \ref{prop:reduction}.}
Assume first that \eqref{aaa} holds. Restricting the estimate to an interval $J\subset I$ of length $|J|\sim2^{-k}$ gives \eqref{sch-local}. Proposition~\ref{lem:local equiv} therefore yields
\[
\Big\|\sum_j\lambda_j|S_{\fa}^\zeta P_kf_j|^2\Big\|_{Y(J)}
\lesssim2^{2ks}\|\lambda\|_{\ell^\beta}
\]
which is \eqref{e:lolo1}, uniformly in $k$ and $J$. Proposition \ref{lem:glob3} then gives \eqref{bbb} on $I$.

Conversely, suppose that \eqref{bbb} holds. After restriction to
each interval $J$ of length comparable to $2^{-k}$, the reverse implication in Proposition \ref{lem:local equiv} gives the corresponding free Schr\"odinger
estimate on $J$. We then apply Proposition \ref{lem:glob3} with
$\zeta(x,t)=x$ and $\fa(t)=(t,\ldots,t)$, so that
$S_{\fa}^\zeta=e^{it\Delta}$. This yields \eqref{aaa} on $I$ and completes the proof. \qed\\

In the remainder of this section, we prove Proposition~\ref{lem:local equiv}. The proof of Proposition~\ref{lem:glob3} is given in Section \ref{sec:5}.

\subsection{Equivalence in superlocal temporal estimates}

Following the single function argument in \cite[Lemma 4.2]{CLV}, we first establish a representation for each function $f$ that preserves orthonormality.

\begin{lem}\label{lem:ap}
Let $k\ge1$, and let $J_0\subset I$ be an interval centered at $t_0$ with $|J_0|\sim 2^{-k}$. 
Assume that $\fa$ and $\zeta$ satisfy \eqref{deriv-a}, \eqref{lip-const}, and \eqref{zeta-holder}.
Then there exist a function $\varphi \in C_0^\infty(\R^d)$ supported in a fixed annulus, and for each $m\in \Z^d$, a linear isometry $Q_m:L^2(\R^d)\to L^2(\R^d)$ and a coefficient $c_m(x,t)$, such that, for every $f\in L^2(\R^d)$, $x\in B_1^d$ and $t\in J_0$,
\begin{align}\label{approx}
(S_{\fa}^\zeta  P_{k}f)(x,t)
= \sum_{m\in \mathbb Z^d} c_m(x,t) \big[e^{it\Delta} P_{\varphi,k} (Q_m f)\big]
\big( {M}_\fa \zeta_{t_0}(x)\big).
\end{align}
Here, $P_{\varphi,k}$ is defined by $\mathcal F(P_{\varphi,k}g)=\varphi(2^{-k}\cdot)\widehat g$, and $M_\fa=\operatorname{diag}((\fa_l'(t_0))^{-1/2})_{l=1}^d$. Moreover, for any $N\ge1$, $|c_m(x,t)| \lesssim_N (1+|m|)^{-N}$ for all $m\in \Z^d$ with an implicit constant independent of $k$, $t\in J_0$, $x\in B_1^d$.
\end{lem}

\begin{proof}
For simplicity, we set $S_k := S^\zeta_{\fa}P_{k}$.  Changing variables $\xi\rightarrow 2^k\xi$, we obtain
\begin{equation}\label{SP_k}
S_kf(x,t)
=(2\pi)^{-d}\,2^{dk} \int e^{i2^k\zeta_t(x)\cdot \xi-i2^{2k}
\sum_l \fa_l(t)\xi_l^2} \phi(\xi) \widehat f(2^k\xi)\,d\xi.
\end{equation}
By rewriting the functions $\zeta_t(x)$ and $\fa_l(t)$ in the phase as
\begin{align*}
\zeta_t(x)&=\zeta_{t_0}(x)+\big(\zeta_t(x)-\zeta_{t_0}(x)),\\
\fa_l(t)&=\fa_l(t_0) + \fa_l'(t_0)(t-t_0) + \big(\fa_l(t)-\fa_l(t_0)
-\fa_l'(t_0)(t-t_0)\big),
\end{align*}
the phase function in \eqref{SP_k} can be written as
\begin{equation*}
\begin{aligned}
&2^k\zeta_t(x)\cdot \xi - 2^{2k} \sum_{l=1}^d \fa_l(t)\xi_l^2 \\ 
= & 2^k\zeta_{t_0}(x)\cdot \xi
-2^{2k}  \sum_{l=1}^d \big(\fa_l(t_0)+ \fa_l'(t_0)(t-t_0)\big)\xi_l^2 + 2^kE_1-2^{2k} E_2,
\end{aligned}
\end{equation*}
where
\begin{align*}
E_1(x,t,\xi)&:=\big( \zeta_t(x)-\zeta_{t_0}(x) \big)\cdot \xi,\\
E_2(x,t,\xi)&:=\sum_{l=1}^d\big(\fa_l(t)-\fa_l(t_0)-\fa_l'(t_0)(t-t_0)\big) \xi_l^2.
\end{align*}

Choose a smooth function $\widetilde\phi\in C^\infty_0(B_3^d)$ such that $\widetilde\phi=1$ on $\supp\phi$.
On $[-\pi,\pi]^d$, the following function has the Fourier series expansion
\begin{align}\label{exp-error}
\widetilde \phi(\xi) e^{i(2^kE_1(x,t,\xi)-2^{2k} E_2(x,t,\xi))}
=\sum_{m\in\mathbb Z^d} C_m(x,t)e^{im\cdot \xi}
\end{align}
with $|C_m(x,t)|\lesssim_N (1+|m|)^{-N}$ for any $N\ge1$, uniformly in $k$, $J_0$, $x\in B_1^d$ and $t\in J_0$. Indeed, the Lipschitz condition \eqref{zeta-holder} and $|t-t_0| \lesssim 2^{-k}$ give 
$|\partial_\xi^a  E_1(x,t, \xi)|\lesssim 2^{-k}$ for all multi-indices $a\in\N^d_0$ and $\xi\in\supp \widetilde\phi$.
Similarly, Taylor's theorem and $\fa_l \in C^2(I)$ give $ |\partial_\xi^a E_2(x,t,\xi)|\lesssim 2^{-2k}$
for all $a \in \mathbb N_0^d$ on the support of $\widetilde \phi$. The Fourier coefficients in \eqref{exp-error} are $C_m(x,t)=(2\pi)^{-d}\int e^{i(-m\cdot \xi +2^kE_1-2^{2k}E_2)} \widetilde\phi(\xi)\,d\xi$. The desired decay estimates follow by repeated integration by parts.

We now insert \eqref{exp-error} into \eqref{SP_k} and obtain
\begin{align*}
S_kf(x,t)
&=\sum_{m\in\mathbb Z^d} C_m(x,t) \frac{2^{dk}}{(2\pi)^d} \int e^{i(2^k\zeta_{t_0}(x)\cdot \xi-2^{2k}t\sum_{l=1}^d 
\fa_l'(t_0)\xi_l^2)}\phi(\xi)\widehat {f_m}(2^k\xi)\,d\xi.
\end{align*}
Here, $f_m$ is given by
\begin{align*}
\widehat {f_m}(\xi)
= e^{i\Psi_m(2^{-k}\xi)} \widehat f(\xi),
\end{align*}
with $
\Psi_m(\xi)=m\cdot \xi - 2^{2k}\sum_{l=1}^d(\fa_l(t_0)-t_0 \fa_l'(t_0))\xi_l^2.
$
Let $M_\fa=\operatorname{diag}((\fa_l'(t_0))^{-1/2})_{l=1}^d$ and $\varphi:=\phi\circ M_\fa \in C^\infty_0(\R^d)$, whose support lies in a fixed annulus. Thus, changing variables $\xi \rightarrow 2^{-k} M_\fa \xi$, we have
\[
S_kf(x,t)
=\sum_{m\in \mathbb Z^d} C_m(x,t) (\det M_\fa)^{1/2}
e^{it\Delta} P_{\varphi,k} Q_m f(M_\fa\zeta_{t_0}(x))
\]
where 
$
\widehat {Q_m f}(\xi)= (\det M_\fa)^{1/2} \widehat{f_m}(M_\fa \xi).
$
Clearly, $\|Q_mf\|_2=\|f\|_2$.
Finally, setting $c_m= (\det M_\fa)^{1/2}C_m$ proves \eqref{approx} and the asserted coefficient bounds.
\end{proof}

Using Lemma \ref{lem:ap}, we now prove Proposition \ref{lem:local equiv}.

\begin{proof}[Proof of Proposition \ref{lem:local equiv}] 
It suffices to consider sufficiently large $k$, since the finitely many remaining terms follow directly from Bessel's inequality.
We first prove $(i)\Rightarrow(ii)$.
We treat the first mixed norm $Y(J_0)=L_x^{\frac q2}(B_1^d; L_t^{\frac r2}(J_0))$ only, because the argument for the other mixed norm is the same. By Lemma \ref{lem:ap}, we have \eqref{approx}
for some $\varphi\in C_0^\infty(\R^d)$. Note that $c_m$ and $\varphi$ are independent of  the choice of $f$. 
Since $|c_m|\le C_N(1+|m|)^{-N}$ for large $N$, the Cauchy--Schwarz inequality yields, for every $j$,
\[ 
|S_{\fa}^\zeta  P_{k}
 f_j(x,t)|^2
\lesssim \sum_{m\in \mathbb Z^d} (1+|m|)^{-N} \big|e^{it\Delta} P_{\varphi,k} (Q_m f_j)
\big( {M}_\fa \zeta_{t_0}(x)\big)\big|^2. 
\]
Since $\zeta_{t_0}$ is bi-Lipschitz on $B_1^d$, we have
$\{ M_\fa \zeta_{t_0}(x) :x\in B_1^d\}\subset B_C^d(x_0)$ for $x_0=M_\fa\zeta_{t_0}(0)\in \R^d$ and some constant $C>0$. Thus, by the Minkowski integral inequality and the change of variables $x\rightarrow (M_{\fa} \zeta_{t_0})^{-1}x$, we have
\begin{equation}\label{UtoE}
\begin{aligned}
&\Big\| \sum_j \lambda_j |S^\zeta_{\fa} P_{k} f_j|^2 \Big\|_{L_x^{\frac q2}(B_1^d; L_t^{\frac r2}(J_0))}\\
\lesssim 
&\sum_{m\in \mathbb Z^d} (1+|m|)^{-N}
\Big\| \sum_j |\lambda_j| |e^{it\Delta}  P_{\varphi,k} Q_m f_j|^2 \Big\|_{
L_x^{\frac q2}(B_C^d(x_0); L_t^{\frac r2}(J_0))}.
\end{aligned}
\end{equation}
For sufficiently large $k$, the support of $\varphi$ gives $P_{\varphi,k}=\sum_{|\ell-k|\le L}P_{\ell}P_{\varphi,k}$ with $\ell\ge1$ and $L$ independent of $k$.
Applying Lemma \ref{lem:TB-T1} to \eqref{sch-local} at each scale $\ell$, with $Q=P_{\varphi,k}Q_m \in \mathcal B(L^2(\R^d))$, and using translation invariance and finite covering of $B_C^d(x_0)$, we obtain
\[
\Big\|\sum_j|\lambda_j||e^{it\Delta}P_{\varphi,k}Q_mf_j|^2\Big\|_{L_x^{q/2}(B_C^d(x_0);L_t^{r/2}(J_0))}
\lesssim 2^{2ks}\|\lambda\|_{\ell^\beta}.
\]
Thus, summing over $m$ in \eqref{UtoE} yields \eqref{aux-local}, proving $(i)\Rightarrow(ii)$.

For the reverse implication ($(ii)\Rightarrow(i)$), set $\psi=\phi\circ M_\fa^{-1}$.
Repeating the Fourier-series argument above in \eqref{exp-error}, we obtain
\begin{align*}
e^{it\Delta}P_kf(y)
=
\sum_{m\in\Z^d}\widetilde c_m(y,t)
\big[S_{\fa}^{\zeta}P_{\psi,k}(\widetilde Q_mf)\big]
\big((M_{\fa}\zeta_{t_0})^{-1}y,t\big)
\end{align*}
for $y \in (M_{\fa} \zeta_{t_0})(B_1^d)$ and $t\in J_0$. 
Here, $\psi\in C_0^\infty(\R^d)$ is supported in a fixed annulus, and each $\widetilde Q_m$ is an $L^2$-isometry, and
$|\widetilde c_m(y,t)|\lesssim_N(1+|m|)^{-N}$.
By change of variables $y \rightarrow (M_\fa \zeta_{t_0})y$, $B_c^d(x_0)\subset M_\fa \zeta_{t_0}(B_1^d)$ for a constant $c>0$. By translation invariance of $e^{it\Delta}P_k$, we cover the ball $B_1^d$ by small balls $B_c^d(x_0)$. This proves $(ii)\Rightarrow(i)$.
\end{proof}

\section{From superlocal to local estimates}\label{sec:5}

In this section, we prove a globalization result in a slightly more general form than is needed for Theorem \ref{thm:equiv}, allowing both an arbitrary temporal scale $\nu^\theta$, with $\theta>1$, and an $\alpha$-dimensional measure in the spatial variable. The case used in the proof of Theorem \ref{thm:equiv} corresponds to $\theta=2$ and Lebesgue measure, for which $\alpha=d$.

Fix $\theta>1$ and $0<\alpha \le d$. A positive Borel measure $\mu$ on $\R^d$ is called $\alpha$-dimensional if there exists a constant $C_\mu>0$ such that
\[
\mu(B_\rho^d(x)) \le C_\mu \rho^\alpha, \quad \text{for all} \quad~ (x,\rho)\in \R^d\times (0,\infty).
\]
For $0<\nu\ll1$, let $T_\nu$ be an integral operator mapping $L^2(\R^d)$ into a space of measurable functions on $\R^d\times \R$. We assume that, for some $C_0>0$, the kernel $K_\nu$ of $T_\nu T_\nu^*$ satisfies 
\begin{equation}\label{e:TT* hypothesis2}
    |K_\nu(x,t,x',t')|\lesssim_N \,\psi^N_{\nu,\theta}(x-x',t-t'),
\end{equation}
for every integer $N\ge1$ and all  $(x,t),\,(x',t')\in B_1^d\times I$ such that $|t-t'|\ge C_0 \nu^{\theta-1}$, where
\begin{equation*}
\psi^N_{\nu,\theta}(x,t)=\nu^{-d}(1+\nu^{-1}|x|)^{-N}(1+\nu^{-\theta}|t|)^{-N},\quad (x,t)\in \R^d\times \R.
\end{equation*}

Proposition \ref{lem:glob2} shows that, for $1\le \beta \le2$ and under the mixed norm conditions below,
a superlocal estimate on intervals of length $\sim \nu^{\theta-1}$ can be globalized to $I$.

\begin{prop}\label{lem:glob2}
Let $q,r\in [2,\infty]$ and $1\le \beta \le 2$. For each interval $E \subset I$, let $Y(E)$ be either $Y(E)=L_x^{\frac q2}(B_1^d,d \mu;L_t^{\frac r2}(E,dt))$ or $Y(E)=L_t^{\frac r2}(E,dt;L_x^{\frac q2}(B_1^d,d\mu))$.
Let $\mu$ be an $\alpha$-dimensional measure in $\R^d$. Suppose
\begin{equation}\label{cond beta r}
\begin{aligned}
2\beta\le
\begin{cases}
\min\{q,r\},&Y(I)=L_x^{\frac q2}(B_1^d,d \mu;L_t^{\frac r2}(I,dt)),\\
r,&Y(I)=L_t^{\frac r2}(I,dt;L_x^{\frac q2}(B_1^d,d\mu)).
\end{cases}
\end{aligned}
\end{equation}
For $0<\nu\ll 1$, let $\mathcal J:=\{J\}$ be a finite partition of $I$ into disjoint intervals of equal length $\sim \nu^{\theta-1}$, and suppose that $T_\nu$ satisfies \eqref{e:TT* hypothesis2}.
Assume further that $C(\nu)>0$ decays at most polynomially as $\nu \to 0$, and that, for every $J\in \mathcal J$, 
    \begin{equation}\label{H1}
        \Big\|\sum_j \lambda_j |T_\nu f_j|^2\Big\|_{Y(J)}
        \le
        C(\nu) \|\lambda\|_{\ell^\beta} 
    \end{equation}
    for all orthonormal family $(f_j)_j\subset L^2$ and all sequences $\lambda = (\lambda_j)_j \in \ell^\beta$. Then
    \begin{equation}\label{H2}
        \Big\|\sum_j \lambda_j |T_\nu f_j|^2\Big\|_{Y(I)}
        \lesssim
        C(\nu) \|\lambda\|_{\ell^\beta} 
    \end{equation}
holds for all orthonormal family $(f_j)_j \subset L^2$ and all sequences $\lambda = (\lambda_j)_j\in \ell^\beta$.
\end{prop}

\subsection{Kernel estimates}\label{sec:local S}

We now prove Proposition \ref{lem:glob3}. We first verify the kernel bound \eqref{e:TT* hypothesis2} for $S_{\fa}^\zeta P_k$. We have
\[
(S_{\fa}^\zeta P_k)(S_{\fa}^\zeta P_k)^*G(x,t) = (2\pi)^{-d}\iint_{B_1^d\times I}  \mathcal K_k(x,t,x',t')G(x',t')\,dx'dt',
\]
where
\[
\CK_k(x,t,x',t')=\int e^{i(\zeta_t(x)-\zeta_{t'}(x'))\cdot \xi - i\sum_{l=1}^d(\fa_l(t)-\fa_l(t'))\xi_l^2} \, \phi^2(2^{-k}\xi)\,d\xi.
\]
We shall use the following kernel estimates.
\begin{lem}\label{lem:ker}
Let $x,x'\in B_1^d$, $t,t'\in I$ and $k\ge1$, and let $A_0,D_0$ denote the constants in \eqref{deriv-a}, \eqref{lip-const}, \eqref{zeta-holder}. Then there exists a constant $C_0 =C_0 (A_0,D_0)>0$ such that,
whenever $|t-t'|\ge C_0 2^{-k}$, for any $N\ge1$
there is $C_N>0$ satisfying
\begin{align*}
|\CK_k(x,t,x',t')|\le C_N 2^{dk}(1+2^k|x-x'|)^{-N}
(1+2^{2k}|t-t'|)^{-N}.
\end{align*}
\end{lem}
The proof follows the nonstationary phase method in \cite[proof of Lemma 2.1]{CLV}.

\begin{proof}
Recall that $\phi$ is smooth and supported where $1/2\le|\xi|\le2$.
By changing variables $\xi \rightarrow 2^k \xi$, we write
\[\CK_k(x,t,x',t')=2^{dk}\int e^{i2^{2k}|t-t'|\Psi_k(x,t,x',t',\xi)} \phi^2(\xi)\,d\xi,
\]
where
\[
\Psi_k(x,t,x',t',\xi)=\frac{(\zeta_t(x)-\zeta_{t'}(x'))\cdot \xi}{2^k|t-t'|} - \frac{\sum_{l=1}^d(\fa_l(t)-\fa_l(t'))\xi_l^2}{|t-t'|}.
\]
We claim that
\begin{align}\label{e:Psik}
&    |\nabla_\xi \Psi_k|\ge A_0/2 \qquad
  |\partial_\xi^a \Psi_k|\le C_a,
\end{align}
holds for any multi-index $a$, uniformly on the support of $\phi$. Indeed, by \eqref{deriv-a} and the mean value theorem,
$|\fa_l(t)-\fa_l(t')| \ge A_0|t-t'|$, $1\le l \le d$.
Consequently for $\xi \in \supp \phi$,
\[
\big|\nabla_\xi \big(\sum_{l=1}^d(\fa_l(t)-\fa_l(t'))\xi_l^2\big)\big| \ge A_0|t-t'|.
\]
Thus, after division by $|t-t'|$, the gradient of the quadratic part of $\Psi_k$
is bounded below by $A_0$ on $\supp\phi$.

On the other hand, by \eqref{lip-const} and \eqref{zeta-holder}, we have
$|\zeta_t(x)-\zeta_{t'}(x')| 
\le D_0(|x-x'|+|t-t'|)\le 3D_0$, where we used $x,x'\in B_1^d$, $t,t'\in I$.
Hence if $|t-t'|\ge C_0 2^{-k}$ and $C_0$ is chosen sufficiently large in terms of $A_0$, $D_0$, then
\[
\frac{|\zeta_t(x)-\zeta_{t'}(x')|}
{2^k|t-t'|}
\le
\frac{A_0}{2}.
\]
This proves the lower bound in \eqref{e:Psik}. The derivative bounds in \eqref{e:Psik} are immediate, since $\Psi_k$ is a polynomial of degree $2$ in $\xi$.

Since $|\nabla_\xi \Psi_k|\gtrsim 1$ on the support of $\phi$, repeated integration by parts yields
\begin{align*}
|\mathcal K_k(x,t,x',t')| 
&\lesssim_N 2^{dk}(1+2^{2k}|t-t'|)^{-2N}
\end{align*}
for all $N\ge1$ provided $|t-t'|\ge C_0 2^{-k}$.
Finally, since $x,x'\in B_1^d$ and $|t-t'|\ge C_0 2^{-k}$, we have
$2^{-1}C_0 2^{k}|x-x'|\le 2^{2k}|t-t'|$. Therefore,
\begin{align*}
|\mathcal K_k(x,t,x',t')| 
&\lesssim_N 2^{dk}(1+2^k|x-x'|)^{-N}(1+2^{2k}|t-t'|)^{-N}.
\end{align*}
This completes the proof.
\end{proof}

We now deduce Proposition~\ref{lem:glob3} from Proposition~\ref{lem:glob2}.

\begin{proof}[Proof of Proposition \ref{lem:glob3}]
The finitely many $k$ for which $2^{-k}\sim1$ are handled by covering $I$ with finitely many intervals $J$ of length $|J|\sim2^{-k}$ and applying \eqref{e:lolo1} together with the triangle inequality. It therefore suffices to consider $0<2^{-k}\ll1$.

By Lemma \ref{lem:ker}, the kernel $\mathcal K_k$ of $(S^\zeta_{\fa}P_k)(S^\zeta_{\fa}P_k)^*$ satisfies
\[
|\mathcal K_k(x,t,x',t')|\lesssim_N\psi^N_{2^{-k},2}(x-x',t-t')
\]
for every $N\ge1$ whenever $|t-t'|\ge C_0 2^{-k}$.
Applying Proposition \ref{lem:glob2} with $\nu=2^{-k}$, $\theta=2$, and $C(\nu)=C\nu^{-2s}=C2^{2ks}$ gives \eqref{e:glgl1}.
\end{proof}

\subsection{Proof of Proposition \ref{lem:glob2}} We prove the result for $Y(I)=L_x^{\frac q2}(B_1^d,d\mu;L_t^{\frac r2}(I,dt))$. The other mixed norm $Y(I)=L_t^{r/2}(I,dt;L_x^{q/2}(B_1^d,d\mu))$ is treated in Section \ref{subsec:time norm}.
Set $T=T_\nu $ and let $\nu^{\theta-1}$ be the length of the intervals in $\mathcal J$. Let $\tilde q, \tilde r \in[2,\infty]$ be defined by the relations \eqref{tilde-qr}.
By Lemma \ref{lem:dual}, applied with spatial measure $d\mu$, the hypothesis \eqref{H1} and the desired estimate \eqref{H2} are equivalent to \textbf{(D-J)} and \textbf{(D-I)} below, respectively.
\begin{enumerate}
	\item \textbf{(D-J)} For every $J\in \mathcal J$ and every $W\in L_x^{\tilde q}(d\mu)L_t^{\tilde r}(dt)$ supported on $B^d_1\times J$,
\begin{align*} 
\big\| WTT^* \overline W \big\|_{\mathcal C^{\beta'}} \le C(\nu) \| W \|^2_{L_x^{\tilde q}(d\mu;L_t^{\tilde r}(dt))}.
\end{align*}
\item\textbf{(D-I)} For every $W\in L_x^{\tilde q}(d\mu)L_t^{\tilde r}(dt)$ supported on $B^d_1\times I$,
\begin{align*} 
\big\| WTT^*\overline W \big\|_{\mathcal C^{\beta'}} \lesssim C(\nu) \| W \|^2_{L_x^{\tilde q}(d\mu;L_t^{\tilde r}(dt))}.
\end{align*}
\end{enumerate}

For a function $W$ as in \textbf{(D-I)}, define
\[
U_J:=\chi_J WT.
\]
By H\"older's inequality and the single function case of \eqref{H1}, each $U_J:L^2(\R^d)\rightarrow L^2(B_1^d\times I;d\mu dt)$ is bounded since $\chi_J T:L^2(\R^d)\to L_x^{q}(B_1^d, d\mu; L_t^{r}(J,dt))$ and $W\in L_x^{\tilde q}(B_1^d, d\mu; L_t^{\tilde r}(I,dt	))$.
Since $U_J(U_{J'})^*=\chi_J WTT^* \overline W \chi_{J'}$, we have
\begin{equation}\label{eq_decom}
   WTT^* \overline W = \sum_{J,J'\in \mathcal J} U_JU_{J'}^*. 
\end{equation}
We decompose $WTT^* \overline W = I_{\mathrm{diag}} + I_{\mathrm{near}} + I_{\mathrm{far}}$, where
\begin{align*}
I_{\mathrm{diag}} = \sum_{J\in \mathcal J} U_JU_J^*,
\quad
I_{\mathrm{near}} = \sum_{\substack{J,J'\in \mathcal J, \, J\neq J',\\\dist(J,J')\le C_0 \nu^{\theta-1}}} U_JU_{J'}^*,
\quad
I_{\mathrm{far}} = \sum_{\dist(J,J')>C_0 \nu^{\theta-1}} U_JU_{J'}^*.
\end{align*}

It remains to prove, for each $\mathcal I\in \{I_{\mathrm{diag}},I_{\mathrm{near}},I_{\mathrm{far}}\}$, that
\begin{equation}\label{eq_goal}
    \big\| \mathcal I\big\|_{\mathcal C^{\beta'}} \lesssim C(\nu) \|W\|^2_{L_x^{\tilde q}(d\mu;L_t^{\tilde r}(dt))}.
\end{equation}
Summing these three bounds and using the triangle inequality in $\mathcal C^{\beta'}$ gives \textbf{(D-I)}, hence \eqref{H2}. All implicit constants are independent of $\nu$ and $W$. 

\subsection{Estimate for $I_{\mathrm{diag}}$}
To control $I_{\mathrm{diag}}$, we use the following property of mutually orthogonal operators.
\begin{lem}[{\cite{Car09}}]\label{lem:car}
    Let $p \in [1,\infty]$. Suppose that a family of operators $\{T_j\}$ on a Hilbert space is mutually orthogonal in the sense that  $T_j T_k^*  =  T_j^* T_k = 0$ for every $j\neq k$. Then
    \[\Big\|\sum_j T_j\Big\|_{\mathcal C^p} =\big(\sum_j \|T_j\|_{\mathcal C^p}^p\big)^{1/p}\]
    where the right-hand side is interpreted as a supremum when $p=\infty$.
\end{lem}

When $J \neq J'$, $\chi_J \chi_{J'}=0$ and hence $U_J^*U_{J'}=0$. 
Hence
\begin{align}\label{disj UU*}
(U_{J}U_{J}^*)(U_{J'} U_{J'}^*)^*=(U_J U_J^*)^*(U_{J'}U_{J'}^*)=0.
\end{align}
Thus, Lemma \ref{lem:car}, applied with $p=\beta'$, gives
\begin{align*}
\big\| I_{\mathrm{diag}} \big\|_{\mathcal C^{\beta'}}
=\Big(\sum_{J\in \mathcal J} \big\| U_J U_J^* \big\|_{\mathcal C^{\beta'}}^{\beta'}\Big)^{\frac 1{\beta'}}
=\Big(\sum_{J\in\mathcal J}\|U_J\|_{\mathcal C^{2\beta'}}^{2\beta'}\Big)^{1/\beta'}.
\end{align*}
Applying \textbf{(D-J)} to each $J$, we obtain
\[
\big\| I_{\mathrm{diag}} \big\|_{\mathcal C^{\beta'}}
\le C(\nu)\Big(\sum_{J\in \mathcal J}\big\| W \chi_J\big\|_{L_x^{\tilde q}(d\mu)L_t^{\tilde r}(dt)}^{2\beta'}\Big)^{\frac1{\beta'}}.
\]
The assumptions $\beta\le q/2$ and $\beta \le r/2$ imply $2\beta'\ge \tilde q$ and $2\beta'\ge \tilde r$, respectively.  Hence, Minkowski's inequality, followed by the embedding $\ell^{\tilde r} \subset \ell^{2\beta'}$, yields
\begin{equation}
\begin{aligned}\label{W-sum}
\Big(\sum_J\big\| W \chi_J\big\|_{L_x^{\tilde q}L_t^{\tilde r}}^{2\beta'}\Big)^{\frac1{\beta'}}
&\le \Big\| \Big(\sum_J\big\| W\chi_J\big\|_{L_t^{\tilde r}}^{2\beta'} \Big)^{\frac1{2\beta'}} \Big\|_{L_x^{\tilde q}}^2\\
&\le \Big\| \Big(\sum_J\big\| W\chi_J\big\|_{L_t^{\tilde r}}^{\tilde r} \Big)^{\frac1{\tilde r}} \Big\|^2_{L_x^{\tilde q}}
=\|W\|_{L_x^{\tilde q}(d\mu)L_t^{\tilde r}(dt)}^2.
\end{aligned}
\end{equation}
Combining these estimates proves \eqref{eq_goal} for $\mathcal I=I_{\mathrm{diag}}$.

\subsection{Estimate for $I_{\mathrm{near}}$}

For each fixed $J\in \mathcal J$ and $m\in\mathbb Z$, define
\[
J_m=J_m(J):=I\cap (J+m\nu^{\theta-1})
=\{t+m\nu^{\theta-1}\in I:t\in J\}.
\]
Then $J_m$ is either an element of $\mathcal J$ or the empty set.
Thus, with
\[
\mathcal M=\{m\in\mathbb Z:0<|m|\le 1+C_0 \},
\]
the near-diagonal term can be written as
\[
I_{\mathrm{near}}
=
\sum_{m\in\mathcal M}\sum_{J\in\mathcal J}
U_JU_{J_m}^*,
\]
where the terms with $J_m=\emptyset$ are omitted.

Fix $m\in\mathcal M$. If $J\neq J'$, then similar to \eqref{disj UU*}, $(U_JU_{J_m}^*)(U_{J'}U_{J'_m}^*)^*
=0$ and $(U_JU_{J_m}^*)^*(U_{J'}U_{J'_m}^*)=0$. Hence the family $\{U_J U_{J_m}^*\}_{J}$ is mutually orthogonal. The triangle inequality and Lemma \ref{lem:car} now give
\begin{equation}\label{eq_tem1}
\big\| I_{\mathrm{near}} \big\|_{\mathcal C^{\beta'}}
\le \sum_{m\in \mathcal M} \Big(\sum_J \big\| U_JU_{J_m}^* \big\|_{\mathcal C^{\beta'}}^{\beta'} \Big)^{\frac 1{\beta'}}.
\end{equation}
By H\"older's inequality for Schatten norms,
\begin{equation}\label{eq_tem}
\big\| U_JU_{J_m}^* \big\|_{\mathcal C^{\beta'}}
\le \big\| U_J \big\|_{\mathcal C^{2\beta'}} \big\|U_{J_m}^* \big\|_{\mathcal C^{2\beta'}}.
\end{equation}
Moreover, $\|U_J\|_{\mathcal C^{2\beta'}}^2 =\|U_J^* U_J\|_{\mathcal C^{\beta'}}=\|U_J U_J^*\|_{\mathcal C^{\beta'}}$ and $\|U_{J_m}^*\|_{\mathcal C^{2\beta'}}^2= \|U_{J_m} U_{J_m}^*\|_{\mathcal C^{\beta'}}$. 
Combining \eqref{eq_tem1}, \eqref{eq_tem} and (\textbf{D-J}), and then applying Cauchy--Schwarz to the sum over $J$,
\begin{align*}
 \big\| I_{\mathrm{near}} \big\|_{\mathcal C^{\beta'}}
&\le C(\nu) \sum_m \Big(\sum_J \big\|W \chi_J\big\|_{L_x^{\tilde q}L_t^{\tilde r}}^{\beta'} \big\|W \chi_{J_m}\big\|_{L_x^{\tilde q}L_t^{\tilde r}}	^{{\beta'}} \Big)^{\frac 1{\beta'}}\\
&\le C(\nu)\sum_m \Big(\sum_J \big\|W\chi_J\big\|_{L_x^{\tilde q}L_t^{\tilde r}}	^{2{\beta'}}\Big)^{\frac1{2\beta'}} \Big(\sum_{J} \big\|W\chi_{J_m}\big\|_{L_x^{\tilde q}L_t^{\tilde r}}^{2{\beta'}}\Big)^{\frac 1{2\beta'}}.   
\end{align*}
By \eqref{W-sum} and injectivity of $J\mapsto J_m$, both $\ell^{2\beta'}$ sums over $J$
are bounded by $\|W\|_{L_x^{\tilde q}(d\mu;L_t^{\tilde r}(dt))}$. Since $\mathcal M$ is finite, we get 
\[\big\| I_{\mathrm{near}}\|_{\mathcal C^{\beta'}}\lesssim 
C(\nu)\|W\|_{L_x^{\tilde q}(d\mu;L_t^{\tilde r}(dt))}^2. \]
This proves \eqref{eq_goal} for $\mathcal I=I_{\mathrm{near}}$. It remains to estimate $I_{\mathrm{far}}$.

\subsection{Estimates for $I_{\mathrm{far}}$} 

The estimate \eqref{eq_goal}  for $\mathcal{I}=I_{\mathrm{far}}$ is simpler because the kernel decays rapidly away from the temporal diagonal, and $W$ is supported on a set $B_1^d\times I$ of finite measure. In fact, we prove the following stronger estimate.

\begin{lem}\label{lem:Ifar} 
Let $2\le \beta' \le \infty$ and $\tilde q, \tilde r\in [2,\infty]$.
For every $N\in \N,$ we have
\begin{equation}\label{e:B}
\| I_{\mathrm{far}}\|_{\mathcal C^{\beta'}} \le C_N \nu^{N} \|W\|_{L^{\tilde q}_x(d\mu; L^{\tilde r}_t(dt))}^2.
\end{equation}
for every $W \in L^{\tilde q}_x(d\mu) L^{\tilde r}_t(dt)$ supported on $B^d_1\times I.$
\end{lem}
Clearly, the estimate \eqref{eq_goal} for $\mathcal{I}= I_{\mathrm{far}}$ follows from Lemma \ref{lem:Ifar} with $N>N_\circ$, since $C(\nu)\gtrsim \nu^{N_\circ}$ for some $N_\circ\in \mathbb N$.

\begin{proof}[Proof of Lemma \ref{lem:Ifar}] 
Since $\mathcal C^2 \subset \mathcal C^{\beta'}$ for $\beta'\ge 2$, it suffices to prove \eqref{e:B} when $\beta'=2$.
Define
\[
E_\nu(x,t,x',t')=\sum_{\substack{J,J'\in \mathcal J,\\\text{dist}(J,J')>C_0\nu^{\theta-1}}}\chi_J(t) K_\nu(x,t,x',t')\chi_{J'}(t'),
\]
where $K_\nu$ is the kernel of $TT^*$. The kernel of $I_{\text{far}}$ is $W(x,t)E_\nu(x,t,x',t')\overline{W(x',t')}$.

We first prove that, for every $N\in\N$, and $(x,t), (x',t')\in B_1^d\times I$,
\begin{equation}\label{e:ptB}
|E_\nu(x,t,x',t')| \lesssim_N \nu^{N}.
\end{equation}
Indeed, let $M=N+d$. If $\dist(J,J')\ge C_0 \nu^{\theta-1}$, $t\in J$ and $t'\in J'$, then $|t-t'|\ge C_0 \nu^{\theta-1}$.
By definition, $\psi_{\nu,\theta}^M(x-x',t-t')\lesssim \nu^{-d}(1+C_0\nu^{-1})^{-M}\lesssim \nu^N$.
Consequently, \eqref{e:TT* hypothesis2} applied with decay order $M$, gives
\begin{equation*}\label{eq_ptbd}
    |\chi_J(t) K_\nu(x,t, x',t')\chi_{J'}(t')|
    \lesssim_N \nu^N\, \chi_J(t)\chi_{J'}(t').
\end{equation*}
Taking the sum over $J$ and $J'$ satisfying $\dist(J,J')> C_0 \nu^{\theta-1}$,  we have \eqref{e:ptB}.

Since the $\mathcal C^2$ norm equals the Hilbert--Schmidt norm,
\[
\| I_{\mathrm{far}}\|_{\mathcal C^2}^2 = \int_{B_1^d\times I}\int_{B_1^d\times I} |W(x,t)E_\nu(x,t,x',t')W(x',t')|^2 d\mu(x')dt' d\mu(x) dt.
\]
It follows from \eqref{e:ptB} and H\"older's inequality that
\begin{equation*}
\| I_{\mathrm{far}}\|_{\mathcal C^2} 
\lesssim \nu^N   \|W\|_{L^{2}(B_1^d\times I, d\mu dt)}^2
\lesssim \nu^N   \|W\|_{L^{\tilde q}_x(d\mu; L^{\tilde r}_t(dt))}^2.
\end{equation*}
The last inequality uses $\tilde q,\tilde r\ge2$. This proves the lemma.
\end{proof}

\subsection{Completion of the proof for $L_t^{r/2}L_x^{q/2}$ norm}
\label{subsec:time norm}
We now prove Proposition~\ref{lem:glob2} for $Y(I)=L_t^{r/2}(I;L_x^{q/2}(B_1^d,d\mu))$.
Here we assume $\beta\le r/2$ only since the intervals $J$ partition $I$ and $L_t^{\tilde r}$ is taken as the outer norm.
Only the summation step \eqref{W-sum} changes.
The assumption $\beta\le r/2$ gives $2\beta'\ge\tilde r$.
Consequently,
\[
\Big(\sum_J\|W\chi_J\|_{L_t^{\tilde r}L_x^{\tilde q}}^{2\beta'} \Big)^{1/\beta'}
\le
\Big(\sum_J\|W\chi_J\|_{L_t^{\tilde r}L_x^{\tilde q}}^{\tilde r} \Big)^{2/\tilde r}=\|W\|_{L_t^{\tilde r}L_x^{\tilde q}}^2.
\]
Unlike \eqref{W-sum}, this argument does not interchange a sequence norm
with a spatial norm, so it requires no relation between $2\beta'$ and
$\tilde q$.
This completes the proof of Proposition \ref{lem:glob2}. \qed

\section{Necessary conditions on the regularity order}\label{sec:example}
In this section, we prove Theorem \ref{thm:ex}. To obtain lower bounds on $s$ for which \eqref{beta-beta} holds for $(f_j)\subset H^s(\R^d)$, we combine the single function example of \cite{Bourgain16} in $l$ variables with the many particle example of \cite{BKS_TLMS} in the remaining $(d-l)$ variables.

\subsection{Two examples}
We first record the two ingredients.

\begin{lem}\label{lem:Bourgain}
Let $l\ge1$. For arbitrarily large $R>1$,
there exist a constant $C>1$, independent of $R$, and 
a function $h_R$ with $\|h_R\|_{L^2(\R^l)}=1$ and $\supp \widehat{h}_R \subset \{\xi'\in \R^l: C^{-1}R\le |\xi'|\le CR\}$, and a measurable set
$\Omega_R\subset B_{1/2}^l$ with $|\Omega_R|\gtrsim 1$ such that,  for each $x'\in \Omega_R$, there exists $\tau_R(x')\in(0,CR^{-1})$ satisfying
\begin{align}\label{hr-low}
 \left|e^{i\tau_R(x')\Delta_l}h_R(x')\right|
 &\gtrsim R^{\frac{l}{2(l+1)}}.
\end{align}
\end{lem}

For $l\ge2$, Lemma \ref{lem:Bourgain} is the normalized form of Bourgain's counterexample in \cite{Bourgain16}. The frequency localization and $L^2$ normalization follow from \cite[(2.2)]{Bourgain16}, while the construction in \cite[(2.3)--(2.9)]{Bourgain16} gives a set of fixed positive measure on which \eqref{hr-low} holds at time of size $O(R^{-1})$. After spatial dilation, this set may be taken inside $B_{1/2}^l$.
The case $l=1$ follows from the classical example of Dahlberg--Kenig \cite{DK82}.

\begin{lem}\label{lem:BKS2}
Let $1\le r \le d$ and $A>0$.  There exist constants $0<\delta<1/2$ and $c_0>0$ such that, for every sufficiently large $R>1$,
the set $J_R = (B_R^r \setminus B_{R/2}^r)\cap \Z^r$ satisfies $|J_R|\sim R^r$, and there exists an orthonormal family $(p_j)$, $j\in J_R$ in $L^2(\R^r)$ such that $\supp \,\widehat{p}_j \subset B_\delta^r(j)$, and
\begin{align}\label{pj-low}
|e^{it\Delta_r}p_j(x'')|\ge c_0
\end{align}
whenever $j\in J_R$, $x''\in B_{1/2}^r$, and $0\le t\le AR^{-1}$.
\end{lem}

A closely related construction is used in \cite{BKS_TLMS}. For completeness, we include the elementary proof.

\begin{proof}
Let $0<\delta<1/2$ be sufficiently small and let $\psi\in \mathcal S(\mathbb R^r)$ satisfy
$\|\psi\|_2=1$, $\widehat\psi\geq0$, and $\supp\widehat\psi\subset B_\delta^r$. 
For $j\in J_R$, define
\[
p_j(x'')=e^{ij\cdot x''}\psi(x'')
\]
so that $\supp \widehat p_j \subset B_\delta^r(j)$. Since $\delta<1/2$, the Fourier supports are pairwise disjoint.
Hence $(p_j)$, $j\in J_R$ are orthonormal in $L^2(\R^r)$ since $\|p_j\|_2=1$.

Since $\widehat p_j=\widehat \psi(\cdot-j)$, we observe that
\begin{align*}
\big|e^{it\Delta_r}p_j(x'')\big|=(2\pi)^{-r}
\Big|\int e^{i(x''-2tj)\cdot \xi''-it|\xi''|^2} \widehat \psi(\xi'')\,d\xi'' \Big|.
\end{align*}
If $x''\in B_{1/2}^r$, $j\in J_R$, and $0\leq t\le AR^{-1}$, then $|x''-2tj|\le \frac12+2A$. By taking $\delta>0$ sufficiently small, \eqref{pj-low} holds provided $x''\in B_{1/2}^r$ and $0\leq t\le AR^{-1}$.
\end{proof}

We now prove Theorem \ref{thm:ex} using Lemmas \ref{lem:Bourgain} and \ref{lem:BKS2}.

\subsection{Proof of Theorem \ref{thm:ex}}
Fix $l\in\{0,1,\ldots,d\}$ and set $r=d-l$.
Let $h_R$, $\Omega_R$, and
$\tau_R$ be as in Lemma \ref{lem:Bourgain}.
Also, let $J_R$ and $(p_j)_{j\in J_R}$ be as in Lemma \ref{lem:BKS2},
applied with $A=C$, where $C$ is the constant in Lemma \ref{lem:Bourgain}.

For $ (x',x'')\in\R^l\times\R^{d-l}$, set
\[
g_j(x',x'')=h_R(x')p_j(x'').
\]
Each $g_j$ has $L^2$ norm $1$. Moreover, the identity $ \widehat g_j(\xi) = \widehat h_R(\xi')\widehat p_j(\xi'')$ shows that their Fourier supports are pairwise disjoint and contained in $\{\xi:|\xi|\sim R\}$, uniformly in $j\in J_R$. Thus $\|g_j\|_{H^s(\R^d)}\sim R^s$, and the normalized functions $f_j:=g_j/\|g_j\|_{H^s(\R^d)}$ form an orthonormal family in $H^s(\R^d)$.

Note that
\begin{equation*}
 e^{it\Delta}f_j(x',x'')
 =\|g_j\|_{H^s(\R^d)}^{-1}
 \bigl(e^{it\Delta_l}h_R\bigr)(x')
 \bigl(e^{it\Delta_{d-l}}p_j\bigr)(x'').
\end{equation*}
Here, $\Delta_l$ and $\Delta_{d-l}$ denote the Laplacians in the $x'$, $x''$ variables, respectively.
For each $x'\in \Omega_R$, choose $t=\tau_R(x')$ as in Lemma \ref{lem:Bourgain}. Since $0<\tau_R(x')<CR^{-1}$,
the lower bound \eqref{pj-low} holds at this time
simultaneously for every $j\in J_R$. Hence \eqref{hr-low} and
\eqref{pj-low} imply
\begin{equation}\label{eq:each-lower}
 \left|
 e^{i\tau_R(x')\Delta}f_j(x',x'')
 \right|
 \gtrsim
 R^{-s+\frac{l}{2(l+1)}}
\end{equation}
for every $j\in J_R$ and $(x',x'')\in \Omega_R \times B_{1/2}^{d-l} \subset B_1^d$. Moreover $|\Omega_R\times B_{1/2}^{d-l}|\gtrsim 1$.

Set $\lambda_j=1$ for every $j\in J_R$. Then $\|(\lambda_j)\|_{\ell^\beta}\sim R^{\frac {d-l}\beta}$. Since the same time
$\tau_R(x')$ is used for all $j$, it follows from
\eqref{eq:each-lower} that
\begin{align*}
\sup_{t\in I}
\sum_{j\in J_R}
\left|e^{it\Delta}f_j(x',x'')\right|^2
&\geq
\sum_{j\in J_R}
\left|e^{i\tau_R(x')\Delta}f_j(x',x'')\right|^2
\gtrsim |J_R| R^{\frac{l}{l+1}-2s}
\end{align*}
on $\Omega_R\times B_{1/2}^r$. Therefore,
\begin{equation*}
\Big\| \sup_{t\in I} \Big| \sum_{j\in J_R}\big|e^{it\Delta}f_j\big|^2 \Big|
\Big\|_{L_x^\beta(B_1^d)}
\gtrsim
R^{(d-l)+\frac{l}{l+1}-2s}.
\end{equation*}

Thus, if \eqref{beta-beta} holds uniformly in $R$, then
$ (d-l)+\frac{l}{l+1}-2s \leq\frac {d-l}\beta$.
This gives
\[
 s
 \geq
 \frac{l}{2(l+1)}
 +
 \frac{d-l}{2}
 \left(1-\frac1\beta\right)
\]
for every $l\in\{0,1,\ldots,d\}$. Hence \eqref{beta-beta} holds only if $s\ge \sigma(d,\beta)$ defined in \eqref{nec range}.
\qed

\section{Remarks on the Strichartz estimates}\label{sec:7}

Although the main focus of this paper is on maximal estimates, the pointwise
identity established in Section~\ref{sub:twisted} also gives an exact
correspondence between Strichartz estimates for the special Hermite and free
Schr\"odinger propagators. This is analogous to the Sj\"ogren--Torrea identity
for the Hermite propagator. We use this observation below to transfer
orthonormal Strichartz estimates from the free Schr\"odinger flow to the
special Hermite flow and, in particular, to recover the result of
Ghosh--Mondal--Swain \cite{GMS}.

\subsection{An identity for the special Hermite flow}

Recall that the special Hermite operator $\mathcal L$ acts on
$\mathbb C^d\simeq\mathbb R^{n}$. We write $n=2d$ for the underlying real
dimension. The following identity is the analogue of the Sj\"ogren--Torrea
identity for the Hermite propagator.

\begin{prop}\label{prop:sp-stri}
Let $d\geq1$ and $n=2d$. Let $2\leq r,q\leq\infty$ satisfy
\[
    \frac{2}{r}+\frac{n}{q}=\frac{n}{2}.
\]
Then, for all orthonormal systems $(f_j)\subset L^2(\mathbb R^n)$
and all sequences $(\lambda_j)$,
\[
\Big\|
    \sum_j\lambda_j |e^{-it\mathcal L}f_j|^2
\Big\|_{L_t^{\frac r2}((0,\frac \pi2);L_z^{\frac q2}(\mathbb R^n))}
=
\Big\|
    \sum_j\lambda_j |e^{is\Delta}f_j|^2
\Big\|_{L_s^{\frac r2}((0,\infty);L_z^{\frac q2}(\mathbb R^n))}.
\]
In particular, taking a single function gives
\[
    \|e^{-it\mathcal L}f\|_{L_t^r((0,\frac \pi2);L_z^q(\mathbb R^n))}
    =
    \|e^{is\Delta}f\|_{L_s^r((0,\infty);L_z^q(\mathbb R^n))}.
\]
\end{prop}

\begin{proof}
Let $s=\tan t$, so that $d t=(1+s^2)^{-1} d s$. By \eqref{id-special},
\[
    \sum_j\lambda_j |e^{-it\mathcal L}f_j(z)|^2
    =
    (1+s^2)^d
    \sum_j\lambda_j
    |e^{is\Delta}f_j((I_{2d}-sJ)z)|^2.
\]
Since $\det(I_{2d}-sJ)=(1+s^2)^d$,
the change of variables $\zeta=(I_{2d}-sJ)z$ gives
\[
\Big\|
    \sum_j\lambda_j |e^{-it\mathcal L}f_j|^2
\Big\|_{L_z^{\frac q2}}
=
(1+s^2)^{d-\frac{2d}{q}}
\Big\|
    \sum_j\lambda_j |e^{is\Delta}f_j|^2
\Big\|_{L_\zeta^{\frac q2}}.
\]
Consequently, after taking the $L_t^{\frac r2}$ norm and the change of variables $s=\tan t$, the remaining power of $1+s^2$ is
\[
    \frac r2\Big(d-\frac{2d}{q}\Big)-1
    =
    \frac r2 \Big(\frac n2- \frac nq-\frac 2r\Big),
\]
but this equals zero under the admissible condition. 
\end{proof}

Proposition~\ref{prop:sp-stri} shows that, on the Schr\"odinger admissible
line, orthonormal Strichartz estimates for the special Hermite operators are
inherited directly from those for the free Schr\"odinger operators. More precisely,
whenever
\[
\Big\|
    \sum_j\lambda_j |e^{is\Delta}f_j|^2
\Big\|_{L_s^{\frac r2}(\mathbb R;L_z^{\frac q2}(\mathbb R^n))}
\lesssim
\|\lambda\|_{\ell^\beta}
\]
holds for all orthonormal systems $(f_j)\subset L^2(\mathbb R^n)$ and all $(\lambda_j)\in \ell^\beta$, we obtain
\[
\Big\|
    \sum_j\lambda_j |e^{-it\mathcal L}f_j|^2
\Big\|_{L_t^{\frac r2}((0,\frac \pi2);L_z^{\frac q2}(\mathbb R^n))}
\lesssim
\|\lambda\|_{\ell^\beta}.
\]

The same estimate holds on every interval of length $\pi/2$. Indeed, let
$I_0=t_0+(0,\pi/2)$ and set $\tau=t-t_0$, $g_j=e^{-it_0\mathcal L}f_j$. Then
\[
\Big\|
    \sum_j\lambda_j |e^{-it\mathcal L}f_j|^2
\Big\|_{L_t^{\frac r2}(I_0;L_z^{\frac q2}(\mathbb R^n))}
=\Big\|
    \sum_j\lambda_j |e^{-i\tau\mathcal L}g_j|^2
\Big\|_{L_\tau^{\frac r2}((0,\frac\pi2);
L_z^{\frac q2}(\mathbb R^n))}
\]
and $(g_j)$ remains orthonormal since $e^{-it_0\mathcal L}$ is unitary on $L^2$. Thus the estimate extends to any fixed bounded time
interval by a finite covering.

\subsection{A simple application}

The preceding identity should be compared with the Hermite case studied by
Bez et al. \cite[Section~7.3]{BHLNS}. Using the Sj\"ogren--Torrea
identity \cite{ST}, they showed that orthonormal Strichartz estimates for the
Hermite propagator follow directly from the corresponding free Schr\"odinger
estimates. Proposition~\ref{prop:sp-stri} gives the analogous implication for
the special Hermite propagator.

\begin{cor}[\cite{GMS}, Theorem 1.6]\label{cor:GMS}
Let $d\geq1$. Suppose that
\[
    2\leq q<\frac{2(2d+1)}{2d-1},
    \qquad
    \frac{1}{r}+\frac{d}{q}=\frac{d}{2}.
\]
Then,
\[
\Big\|
    \sum_j\lambda_j |e^{-it\mathcal L}f_j|^2
\Big\|_{L_t^{\frac r2}((-\pi,\pi);
L_z^{\frac q2}(\mathbb C^d))}
\lesssim
\|\lambda\|_{\ell^{\frac{2q}{q+2}}}
\]
for all orthonormal systems $(f_j)\subset L^2(\mathbb C^d)$ and all
$(\lambda_j)\in\ell^{\frac{2q}{q+2}}$.
\end{cor}

The study of orthonormal Strichartz estimates for the free Schr\"odinger
flow was initiated by Frank--Lewin--Lieb--Seiringer \cite{FLLS} and
further developed in \cite{FS_AJM,BHLNS,BLN_Forum}, with boundary cases
treated in \cite{BKS_TLMS}. 
More recently, Hadama--Toyoshima \cite{HT26} established
endpoint orthonormal Strichartz estimates with Lorentz sequence
norms in dimensions at least two.
Gong \cite{Gong26} obtained critical summability estimates
for the free and fractional Schr\"odinger propagators acting
on systems orthonormal in Sobolev spaces. Related estimates for the wave equation were obtained in \cite{BKS_RIMS}. Orthogonal Strichartz estimates have also been used to prove local smoothing estimates in modulation spaces for fractional Schr\"odinger flows \cite{BIS}.

For the present application, we use the corresponding free estimate
on the scaling critical line $\frac{2}{r}+\frac{n}{q}=\frac{n}{2}$ for $n=2d$, with summability exponent
$\beta=\frac{2q}{q+2}$.
Proposition~\ref{prop:sp-stri} therefore gives the estimate on $(0,\frac\pi2)$, and the preceding time-translation
argument extends it to $(-\pi,\pi)$.
The proof in \cite{GMS} proceeds through restriction estimates
for the special Hermite transform and the associated Schatten class estimates.
Here the result follows directly from the corresponding free Schr\"odinger
estimate through the exact identity in Proposition~\ref{prop:sp-stri}.

\subsection{Initial data in inhomogeneous Sobolev spaces}
\label{subsec:loc sch}

The preceding discussion concerns global Strichartz estimates on the
Schr\"odinger admissible line. Away from this line, one can still obtain
local-in-space-time estimates with Sobolev regularity by combining the known
orthonormal Strichartz estimates for the free Schr\"odinger equation with
Theorem~\ref{thm:equiv}.

More precisely, let $d\geq1$ and $r,q\geq2$ satisfy
\[
    \frac1r\leq \frac d2\Big(\frac12-\frac1q\Big).
\]
Fix $0<T_{\mathcal H}<\pi/4$ and set $I_{\mathcal H}=[0, T_{\mathcal H}]$.
Then, for
\[
    s>\frac d2-\frac dq-\frac2r,
\]
and for $\beta$ in the corresponding range of the orthonormal
Strichartz estimates for the free Schr\"odinger equation, subject to $1\le\beta\le2$ and $2\beta\le r$ as required by Theorem~\ref{thm:equiv}, one has
\[
\Big\|
    \sum_j\lambda_j |e^{-it\mathcal H}f_j|^2
\Big\|_{L_t^{\frac r2}(I_{\mathcal H};
L_x^{\frac q2}(B_1^d))}
\lesssim
\|\lambda\|_{\ell^\beta}
\]
for every orthonormal system $(f_j)\subset H^s(\mathbb R^d)$ and every
$(\lambda_j)\in\ell^\beta$.

The same local-in-space-time estimate holds for the special Hermite
propagator $e^{-it\mathcal L}$ with $d$ replaced by $2d$ and $I_{\mathcal H}$ replaced by
$I_{\mathcal L}=[0,T_{\mathcal L}]$, where $0<T_{\mathcal L}<\pi/2$, and with the free exponent range changed accordingly.
These statements follow from the corresponding free
Schr\"odinger estimates by the same localization argument used in the proof
of Theorem~\ref{thm:on-max-gen}, together with Theorem~\ref{thm:equiv}.
Related results for the $(k,a)$--Laguerre, Dunkl, and Dunkl--Hermite
operators can be found in \cite{MS23,GS}.

\appendix

\section{Verification of the $L_x^\beta$ form of Theorem \ref{thm:max}}
\label{sec:known}

\subsection{Frequency-localized endpoint estimates}

As mentioned in the introduction, the result of \cite{BKS_TLMS} underlying Theorem \ref{thm:max} is explicitly stated only for $q=2$. A careful reading of its proof shows that the case $q=2\beta$ can also be obtained. To justify the form of Theorem \ref{thm:max} used in Section \ref{sec:proof:on-gen}, we record the corresponding frequency-localized estimates.
The projections $P_k$ are those defined in Section \ref{sec:4}.

\begin{lem}[{\cite{BKS_TLMS}}]\label{lem:BKS}
Let $d\ge1$, $s\in \R$, $k\ge1$, and $\beta\in \{2,\infty\}$. Then
\begin{align*}
\Big\| \sup_{t\in I} \Big| \sum_j \lambda_j|e^{it\Delta} P_{k} \langle D\rangle^{-s}  g_j|^2 \Big| \Big\|_{L_x^{\beta} (B^d_1)} \lesssim C_\beta(k) 2^{2k(s(d,\beta)-s)}\|\lambda\|_{\ell^\beta}
\end{align*}
holds for all orthonormal initial data $(g_j)\subset L^2(\R^d)$
and all sequences $(\lambda_j)\in \ell^\beta$, where
$s(d,2)=\tfrac d4$ and $s(d,\infty)=\tfrac d2$, and $C_2(k)=k$, $C_\infty(k)=1$.
\end{lem}

\begin{proof}[Sketch of proof of Lemma \ref{lem:BKS}]
For $k\ge1$, first set $\widetilde U_k=e^{it\Delta}P_k|D|^{-s}$ and $\widetilde T_k=\widetilde U_k\widetilde U_k^*$.
On $B_1^d\times I$, decompose $\widetilde T_k=\sum_{l\ge0}\widetilde T_{k,l}$ as in \cite[(6.13)--(6.15)]{BKS_TLMS}. The estimates there imply the bounds 
\begin{equation}\label{e:BKS}
\begin{aligned}
\big\| W\widetilde T_k \overline W\big\|_{\mathcal C^1} &\lesssim 2^{(d-2s)k} \|W\|_{L_x^2L_t^2}^2,\\
\big\| W\widetilde T_{k,l} \overline W \big\|_{\mathcal C^2} &\lesssim (1+2^{k-l})^{-\frac d2} 2^{(d-2s)k}2^{-\frac d2 l} \|W\|_{L_x^4L_t^2}^2.
\end{aligned}    
\end{equation}
Summing the second estimate in \eqref{e:BKS} over all $l\ge0$ yields
\begin{align*}
\big\| W\widetilde T_k \overline W \big\|_{\mathcal C^2} &\lesssim 
\Big(\sum_{0\le l \le k} (2^{k-l})^{-\frac d2} 2^{(d-2s)k}2^{-\frac d2 l} +\sum_{l>k} 2^{(d-2s)k}2^{-\frac d2 l} \Big)
\|W\|_{L_x^4L_t^2}^2 \\
&\lesssim k2^{(\frac d2-2s)k}\|W\|_{L_x^4L_t^2}^2.
\end{align*}
Now choose a smooth Fourier multiplier $\widetilde P_k$ with uniformly bounded symbol supported where $|\xi|\sim2^k$ and satisfying $P_k\widetilde P_k=P_k$. Set $Q_k=|D|^s\langle D\rangle^{-s}\widetilde P_k$.
Then $U_k:=e^{it\Delta}P_k\langle D\rangle^{-s}=\widetilde U_kQ_k,$ with $\sup_{k\ge1}\|Q_k\|_{\mathcal B(L^2)}<\infty$.
Hence by Lemma \ref{lem:TB-T1}, 
\[
\begin{aligned}
\|WU_kU_k^*\overline W\|_{\mathcal C^p}
=\|WU_k\|_{\mathcal C^{2p}}^2
\le \|Q_k\|_{\mathcal B(L^2)}^2
\|W\widetilde T_k\overline W\|_{\mathcal C^p},
\end{aligned}
\]
for $p=1,2$.
By  Lemma \ref{lem:dual}, these estimates give 
\begin{equation*}
\begin{aligned}
&\big\| \sup_{t\in I} \big| \sum_j \lambda_j|e^{it\Delta} P_k \langle D\rangle^{-s}  g_j|^2 \big| \big\|_{L_x^{\infty}(B_1^d)} \lesssim 2^{(d-2s)k}\|\lambda\|_{\ell^\infty}\\
&\big\| \sup_{t\in I} \big| \sum_j \lambda_j|e^{it\Delta} P_k \langle D\rangle^{-s}  g_j|^2 \big| \big\|_{L_x^{2}(B^d_1)} \lesssim k2^{(\frac d2-2s)k}\|\lambda\|_{\ell^2},
\end{aligned}
\end{equation*}
as claimed.
\end{proof}

\subsection{Interpolation and frequency summation}

We now interpolate the endpoint estimates and sum the frequency localized estimates to obtain the inhomogeneous Sobolev estimate used in Section \ref{sec:proof:on-gen}.

\begin{cor}\label{cor:sch-all}
Let $d\ge1$, $\beta \in [1,\infty]$, and $s>s(d,\beta)$ (cf. \eqref{s0nbeta}). Then
\begin{align}\label{corollary goal}
\Big\| \sup_{t\in I} \Big| \sum_j \lambda_j|e^{it\Delta}  \langle D\rangle^{-s}  g_j|^2 \Big| \Big\|_{L_x^{\beta}(B_1^d)} \lesssim \|\lambda\|_{\ell^\beta},
\end{align}
holds for all orthonormal initial data $(g_j) \subset L^2(\R^d)$ and all sequences $(\lambda_j)\in \ell^\beta$.
\end{cor}
\begin{proof}
Fix $s$ and $\beta$ as in Corollary \ref{cor:sch-all} and choose $s_*$ so that $s(d,\beta)<s_*<s$.
We first obtain the required frequency-localized estimate
\begin{align}\label{claim-S}
\Big\| \sup_{t\in I} \Big| \sum_j \lambda_j \big| e^{it\Delta} P_{k} g_j \big|^2 \Big| \Big\|_{L_x^{\beta}(B_1^d)} \le C  2^{2ks_*}\|\lambda\|_{\ell^\beta}
\end{align}
for all $k\ge0$, and any orthonormal system $(g_j)\subset L^2(\R^d)$, and any $(\lambda_j)\in \ell^\beta$. Note that $f_j=\langle D\rangle^{-s}g_j$ form an orthonormal family in $H^s(\R^d)$.
Once \eqref{claim-S} is proved,
we apply Lemma \ref{lem:TB-T1} and sum over $k\ge0$ to obtain \eqref{corollary goal}.

When $\beta=1$, the single function maximal estimate (Theorem~\ref{t:max schrodinger single}) and the triangle inequality give, for every $\tau>\frac d{2(d+1)}$
\begin{align*}
    &\Big\| \sup_{t\in I} \big|\sum_j \lambda_j |e^{it\Delta} P_k g_j|^2\big| \Big\|_{L^1_x(B^d_1)}
    \lesssim \sum_j |\lambda_j|  \Big\| \sup_{t\in I} \big|e^{it\Delta} P_k g_j\big| \Big\|^2_{L^2_x(B^d_1)}
    \lesssim 2^{2k\tau}\|\lambda\|_{\ell^1}.
\end{align*}
Taking $s=0$ in Lemma \ref{lem:BKS} gives the corresponding estimates at $\beta=2$ and $\beta=\infty$.
Interpolating via Lemma \ref{lem:dual}, we obtain, for every $\varepsilon>0$,
\begin{align*}
\Big\|\sup_{t\in I}\Big|\sum_j\lambda_j
 |e^{it\Delta}P_kg_j|^2\Big|\Big\|_{L_x^\beta(B_1^d)}
\lesssim_\varepsilon
2^{2k(s(d,\beta)+\varepsilon)}\|\lambda\|_{\ell^\beta}.
\end{align*}

Here, interpolation between $\beta=1$ and $\beta=2$ gives $s(d,\beta)=\frac{d}{2(d+1)}(d-\frac{d-1}{\beta})$ for $1\le\beta\le2$, while interpolation between $\beta=2$ and $\beta=\infty$ gives $s(d,\beta)=\frac d2(1-\frac1\beta)$ for $2\le\beta\le\infty$. Since $\varepsilon>0$ is sufficiently small, we obtain \eqref{claim-S} for $k\ge1$. The case $k=0$ follows directly from the Bessel's inequality and $\|\lambda\|_{\ell^\infty} \le \|\lambda\|_{\ell^\beta}$.

For $k\ge1$, choose a smooth Fourier multiplier $\widetilde P_k$ such that $P_k\widetilde P_k=P_k$ whose symbol is supported on $\{C^{-1}2^k\le |\xi|\le C2^k\}$ for some $C>1$. Set $Q_k=2^{ks}\langle D \rangle^{-s}\widetilde P_k$. Then $\sup_{k\ge1}\|Q_k\|_{\mathcal B(L^2)}<\infty$.
Applying Lemma \ref{lem:TB-T1} to \eqref{claim-S}, and treating $k=0$ similarly, yields
\begin{align}\label{ccc}
\Big\| \sup_{t\in I} \Big| \sum_j \lambda_j |e^{it\Delta}P_k\langle D\rangle^{-s}g_j|^2 \Big| \Big\|_{L_x^\beta(B_1^d)}
\lesssim 2^{2k(s_*-s)}\|\lambda\|_{\ell^\beta},\qquad k\ge0.
\end{align}
It remains to sum over the frequency scales $k\ge0$.
Applying Minkowski's inequality as in \eqref{eq:Minkowski}
and using \eqref{ccc} together with $s>s_*$, we obtain
\eqref{corollary goal}, completing the proof.
\end{proof}

\section*{Acknowledgements} 
This work was supported by G-LAMP RS-2024-00443714, RS-2024-00339824 (E. Jeong and H. Ko); JBNU research funds for newly appointed professors in 2024 and RS-2026-25471967 (H. Ko); the Croatian Science Foundation under the project number HRZZ-IP-2022-10-5116 (FANAP) (S. Shiraki). The first and second authors would like to thank Chu-Hee Cho for valuable discussions on the twisted Laplacian.

\section*{AI Statement}
The authors used ChatGPT (OpenAI, ChatGPT-5.5 Pro) to clarify their understanding of Schatten spaces and to assist with language editing and improving the readability of the manuscript.

\end{document}